\documentclass[final]{siamltex}
\usepackage{float}
\usepackage{graphicx,subfigure}
\usepackage{multirow}
\usepackage{diagbox}
\usepackage{cancel}
\usepackage{amsmath}
\usepackage{amssymb}
\usepackage{amsfonts}
\usepackage{nicefrac}
\usepackage{color}
\newcommand{\nabh}{\nabla_{\! h}}
\newcommand{\bfphi}{\mbox{\boldmath$\phi$}}

\newcommand{\hf}{\nicefrac{1}{2}}
\newcommand{\nrm}[1]{\left\| #1 \right\|}
\newcommand{\ciptwo}[2]{\left\langle #1 , #2 \right\rangle_\Omega}
\newcommand{\cipgen}[3]{\left\langle #1 , #2 \right\rangle_{#3}}

\newcommand\dt {{\Delta t}}

\newcommand{\eipx}[2]{\left[ #1 , #2 \right]_{\rm x}}
\newcommand{\eipy}[2]{\left[ #1 , #2 \right]_{\rm y}}
\newcommand{\eipz}[2]{\left[ #1 , #2 \right]_{\rm z}}
\newcommand{\eipvec}[2]{\left[ #1 , #2 \right]_{\Omega}}

\newtheorem{rem}{Remark}[section]
\newtheorem{thm}{Theorem}[section]

\newtheorem{lm}{Lemma}[section]

\begin{document}
\title{Positivity-preserving, energy stable numerical schemes for the
Cahn-Hilliard equation with logarithmic potential}

\author{
Wenbin Chen\thanks{
Shanghai Key Laboratory for Contemporary Applied Mathematics,
School of Mathematical Sciences; Fudan University, Shanghai, China 200433 ({\tt wbchen@fudan.edu.cn})}
\and
Cheng Wang\thanks{Mathematics Department; University of Massachusetts; North Dartmouth, MA 02747, USA ({\tt corresponding author: cwang1@umassd.edu})}
\and
Xiaoming Wang\thanks{Fudan University, Shanghai, China 200433, and Florida State University, Tallahassee, FL 32306, USA ({\tt wxm@math.fsu.edu})}
\and
Steven M. Wise\thanks{Mathematics Department; University of Tennessee; Knoxville, TN 37996, USA ({\tt swise1@utk.edu})}
			}


\maketitle

	\begin{abstract}
In this paper we present and analyze finite difference numerical schemes for the Allen Cahn/Cahn-Hilliard equation with a logarithmic Flory Huggins energy potential. Both the first order and second order accurate temporal algorithms are considered. In the first order scheme, we treat the nonlinear logarithmic terms and the surface diffusion term implicitly, and update the linear expansive term and the mobility explicitly. We provide a theoretical justification that, this numerical algorithm has a unique solution such that the positivity is always preserved for the logarithmic arguments, i.e., the phase variable is always between $-1$ and 1, at a point-wise level. In particular, our analysis reveals a subtle fact: the singular nature of the logarithmic term around the values of $-1$ and 1 prevents the numerical solution reaching these singular values, so that the numerical scheme is always well-defined as long as the numerical solution stays similarly bounded at the previous time step. Furthermore, an unconditional energy stability of the numerical scheme is derived, without any restriction for the time step size. Such an analysis technique could also be applied to a second order numerical scheme, in which the BDF temporal stencil is applied, the expansive term is updated by a second order Adams-Bashforth explicit extrapolation formula, and an artificial Douglas-Dupont regularization term is added to improve the stability property. The unique solvability and the positivity-preserving property for the second order scheme are proved using similar ideas, in which the singular nature of the logarithmic term plays an essential role. For both the first and second order accurate schemes, we are able to derive an optimal rate convergence analysis, which gives the full order error estimate. The case with a non-constant mobility is analyzed as well. We also describe a practical and efficient multigrid solver for the proposed numerical schemes, and present some numerical results, which demonstrate the robustness of the numerical schemes.
	\end{abstract}



\pagestyle{myheadings}
\thispagestyle{plain}
\markboth{W.~CHEN, C.~WANG, X.~WANG AND S.M.~WISE}{POSITIVE, ENERGY STABLE SCHEMES FOR THE CH EQUATION}

	\section{Introduction}
The well-known Allen-Cahn (AC)~\cite{allen79} and Cahn-Hilliard (CH)~\cite{cahn58} equations are prototypical gradient flows with respect to a given free energy. We consider a bounded domain $\Omega \subset \mathbb{R}^d$ (with $d=2$ or $d=3$).  For any $\phi \in H^1 (\Omega)$, with a point-wise bound, $-1 < \phi < 1$, the energy functional is given by
	\begin{equation}
	\label{CH energy}
E(\phi)=\int_{\Omega}\left( ( 1+ \phi) \ln (1+\phi) + (1-\phi) \ln (1-\phi) - \frac{\theta_0}{2} \phi^2 +\frac{\varepsilon^2}{2}|\nabla \phi|^2\right) d {\bf x} ,
	\end{equation}
where $\varepsilon$, $\theta_0$ are positive constants associated with the diffuse interface width. See \cite{cahn1996, elliott92a, doi13, elliott96b}. The AC and CH equations are precisely the $L^2$ (non-conserved) and $H^{-1}$ (conserved) gradient flows of the energy functional (\ref{CH energy}), respectively,
	\begin{equation}
\partial_t \phi = - {\cal M} (\phi) \mu , \qquad \mbox{(Allen-Cahn)}    	
	\label{AC equation-0}
	\end{equation}
and
	\begin{equation}
\partial_t\phi = \nabla \cdot ( {\cal M} (\phi) \nabla \mu ) ,  \qquad \mbox{(Cahn-Hilliard)}
	\label{CH equation-0}
	\end{equation}
where $\mu$ is the chemical potential
	\begin{equation}
\mu := \delta_\phi E = \ln (1+\phi) - \ln (1-\phi) - \theta_0 \phi - \varepsilon^2 \Delta \phi ,
	\label{CH-mu-0}
	\end{equation}
and ${\cal M} (\phi) >0$ is the mobility function. In a related example, Cahn, et al.~\cite{cahn1996} have studied the Cahn-Hilliard equation with the fully degenerate mobility, $\mathcal{M} (\phi) = (1-\phi)(1+\phi)$, and have shown asymptotic convergence to a geometric model for motion by the surface Laplacian of mean curvature.

For simplicity of presentation, we suppose $\Omega$ is a cuboid and consider periodic boundary conditions. The case with homogeneous Neumann boundary condition could be analyzed in a similar manner. Due to the gradient structure of \eqref{AC equation-0} and \eqref{CH equation-0}, the following energy dissipation laws formally hold:
	\begin{align}
\frac{d}{dt} E(\phi(t))= & \ -\int_{\Omega} {\cal M} (\phi) |\mu|^2 d {\bf x} \le 0  \quad \mbox{(AC equation)} ,
	\label{energy-decay-rate-AC}	
	\\
\frac{d}{dt} E(\phi(t))= & \ -\int_{\Omega} {\cal M} (\phi) |\nabla \mu |^2 d {\bf x} \le 0  \quad \mbox{(CH equation)} .
	\label{energy-decay-rate-CH}
	\end{align}

The free energy with the logarithmic potential is often considered to be more physically realistic than that with a polynomial free energy, because the former can be derived from regular or ideal solution theories~\cite{doi13}. However, one well-known difficulty for the analysis of these models with logarithmic Flory Huggins energy potential -- as it is called in the polymer science community~\cite{doi13} -- is associated with the singularity as the phase variable approaches $-1$ or $1$. PDE solutions are expected to satisfy a \emph{positivity property}, specifically,
	\begin{equation}
0 < 1-\phi \quad \mbox{and} \quad  0 < 1+\phi .	
	\end{equation}
In other words, the phase variable remains in the interval $(-1, 1)$, in a point-wise sense~\cite{elliott96b}. However, it is a major challenge to create numerical schemes that mimic this property. To avoid such a subtle challenge, many efforts have been devoted to a polynomial approximation:
	\begin{equation}
E(\phi)=\int_{\Omega}\left( \frac14 (\phi^2 -1)^2 + \frac{\varepsilon^2}{2}|\nabla \phi|^2\right) d {\bf x} ,
	\label{CH energy-polynomial-0}
	\end{equation}
which leads to the nonlinear, but non-singular, chemical potential
	\begin{equation}
  \mu := \delta_\phi E = \phi^3 - \phi- \varepsilon^2 \Delta \phi .
	\label{CH-polynomial-mu-0}
  	\end{equation}
This model has a similar double-well structure as in the  case~\eqref{CH energy} and \eqref{CH-mu-0}, but avoids the singularities as the phase variable approaches $-1$ or $1$. Meanwhile, the PDE solution may go beyond the given interval of $(-1, 1)$. There have been extensive numerical works for the Cahn-Hilliard equation with the polynomial approximation~\eqref{CH energy-polynomial-0}, \eqref{CH-polynomial-mu-0}; see the related references~\cite{christlieb14, du91, elliott93, furihata01, guillen14, LiD2016b, LiD2016a, wu14}, et cetera.

In this article we focus on the Cahn-Hilliard model with logarithmic Flory Huggins energy potential~\eqref{CH energy}. At the PDE level, the positivity property (for both logarithmic arguments, $1 + \phi$ and $1 - \phi$) has been established in~\cite{abels07, debussche95, elliott96b, miranville04}. Moreover, in the 1-D and 2-D cases, the phase separation has also been justified at a theoretical level, i.e., a uniform distance between the phase variable and the singular limit values ($-1$ and $1$) have been derived, dependent on $\varepsilon$, $\theta_0$ and the initial data. The analysis for the degenerate mobility case could be found in~\cite{barrett99, elliott96b}. In addition, an improved analysis for the 2-D equation has been reported in a more recent work~\cite{Giorgini17a}; also see the related references~\cite{miranville11, miranville12}. An extension to the Cahn-Hilliard model coupled with fluid flow is also discussed in~\cite{abels09b, Giorgini17}.

At the level of numerical scheme design, the positivity preserving property is very challenging, due to the particularities of the spatial and temporal discretizations  involved. There have been extensive numerical works for the CH model with Flory Huggins energy potential~\cite{jeong16, jeong15, LiH2017, LiX16, peng17a, peng17b, yang17c}, while a theoretical justification to assure the positivity of $1+\phi$ and $1-\phi$ has not been available (so that the numerical scheme is unconditionally well-defined). Among the existing literature, it is worth mentioning the numerical analysis to theoretically justify this issue in~\cite{elliott92a}. The authors analyzed the implicit Euler scheme applied to the CH equation~\eqref{CH equation-0}, \eqref{CH-mu-0}, combined with the finite element approximation in space. In more details, the following result was proved: {\em Under the condition that $\dt \le \frac{4 \varepsilon^2}{\theta_0^2}$, and the initial data satisfy $\frac{1}{| \Omega |} \left| \int_\Omega \, \phi_0 ({\bf x}) \, d {\bf x} \right| < 1 - \delta$, $\| \phi_0 \|_\infty \le 1$, then there is a unique numerical solution for the fully implicit Euler scheme, satisfying $\| \phi^n \|_\infty < 1$.}  An extension to the multi-component Cahn-Hilliard flow has also been reported in~\cite{elliott96c}.

Meanwhile, it is observed that, an energy stability property is not unconditionally available for the scheme studied in~\cite{elliott92a}, due to the implicit treatment of the concave diffusion term. Further, the time step constraint, $\dt \le \frac{4 \varepsilon^2}{\theta_0^2}$, could make the numerical implementation challenging for small $\varepsilon$ and large $\theta_0$. In this article, we propose and analyze an alternate numerical scheme, in which the  implicit treatment for the concave diffusion term is replaced by an explicit one, combined with centered difference discretization in space. Again, the implicit treatment for the nonlinear logarithmic term requires a theoretical justification for the positivity of both $1+\phi$ and $1-\phi$, so that the numerical scheme is well-defined at a point-wise level. Instead of reconstructing an alternate energy functional to avoid the singularity for $\phi$ at $-1$ and $1$, as reported in~\cite{elliott96c, elliott92a}, we use a new technique to theoretically justify the positivity of the numerical solution. First, the fully discrete numerical scheme corresponds to a minimization of a discrete energy function. And also, such an energy functional is strictly convex, as long as the phase variable stays within $(-1,1)$ at a point-wise level. Subsequently, to avoid a circular ``chicken-and-the-egg" argument, we take a closed domain for the numerical solution variables, in which the limit bound values of $-1$ and $1$ are not reachable. In turn, the continuous energy function has to have a global minimum over this closed domain. Moreover, we make use of the following subtle fact: the singular nature of the logarithmic function prevents such a global minimum from being obtained at a boundary point (in terms of numerical solution variable domain), as long as the numerical solution stays bounded at the previous time step. As a result, since the global minimum could only possibly occur at an interior point in the numerical solution variable domain, we conclude that the numerical scheme has to be satisfied, so that the existence of the numerical solution is proved. In addition, due to the convex nature of the energy function, the uniqueness of the numerical solution becomes a direct consequence. As a further consequence, we observe that: as long as the numerical solution stays bounded at the previous time step, i.e., within $[-M, M]$ ($M >0$), not necessarily $(-1, 1)$, and its average stays between $-1$ and $1$, there must exist a unique numerical solution which stays within $(-1, 1)$ at the next time step. 
This leads to an interesting difference between the present results and those in~\cite{elliott92a}, where the requirement for the initial data, namely, $\| \phi_0 \|_\infty \le 1$, has to be imposed for the analysis to go through. On the other hand, the latter constraint is completely natural.  Another new feature of the numerical analysis in this article is the theoretical justification of the energy stability. As a result of the unconditional energy stability, a uniform in time $H_h^1$ bound for the numerical solution could be derived. In addition, a detailed convergence analysis of the proposed numerical scheme could be derived, which gives an optimal rate error estimate in the $\ell^\infty (0,T; H_h^{-1}) \cap \ell^2 (0, T; H_h^1)$ norm. A key point in the analysis lies in the following subtle fact: since the nonlinear logarithmic term corresponds to a convex energy, the corresponding nonlinear error inner product is always non-negative. And also, the error estimate associated with the surface diffusion term indicates an $\ell^2 (0, T; H_h^1)$ convergence. Because of the explicit treatment for the expansive term, this convergence estimate does not require the time step constraint, in comparison with the existing results~\cite{barrett95, barrett96, barrett01, elliott96c, elliott92a}.

On the other hand, all these positivity-preserving schemes are only first order accurate in time, which is not satisfactory in the practical computations. In turn, a higher order accurate in time, positivity-preserving numerical scheme is highly desired. In this article, we propose and analyze a second order accurate scheme for the CH model with Flory Huggins energy potential~\eqref{CH energy}, with unique solvability, positivity-preserving property and energy stability established. In more details, we apply the implicit backward differentiation formula (BDF) concept to derive second order temporal accuracy, while the expansive term is treated by a second order explicit extrapolation formula. An additional term $A \dt \Delta_h (\phi^{k+1} - \phi^k)$ is added, which represents a second order Douglas-Dupont-type regularization, and a careful calculation shows that energy stability is guaranteed, provided the mild condition $A \ge \frac{1}{16}$ is enforced. Moreover, the singular nature of the logarithmic term enables us to theoretically derive the positivity-preserving property of this second order numerical scheme, which is the first such result in this area. And also, an $H_h^{-1}$ inner product with the numerical error function leads to an optimal rate error estimate in the $\ell^\infty (0,T; H_h^{-1}) \cap \ell^2 (0, T; H_h^1)$ norm, with second order accuracy in both time and space.

The rest of the article is organized as follows. In Section~\ref{sec:numerical scheme} we propose the first order numerical scheme and state the corresponding theoretical results. The detailed proof for the positivity-preserving property of the numerical solution is provided in Section~\ref{sec: proof}. Subsequently, the energy stability analysis is established in Section~\ref{sec:energy stability}, and 
the optimal rate convergence analysis is presented in Section~\ref{sec:convergence}. The second order BDF scheme is outlined and analyzed in Section~\ref{sec:BDF2}. Some numerical results are presented in Section~\ref{sec:numerical results}, including a brief description of the 3-D multigrid solver. Finally, the concluding remarks are given in Section~\ref{sec:conclusion}.

\section{The first order numerical scheme} \label{sec:numerical scheme}

In the spatial discretization, the centered finite difference approximation is applied. We recall some of the basics of this methodology.

	\subsection{Discretization of space}
	\label{subsec:finite difference}

We use the notation and results for some discrete functions and
operators from~\cite{guo16, wise10, wise09}.
Let $\Omega = (0,L_x)\times(0,L_y)\times(0,L_z)$, where for simplicity, we assume $L_x =L_y=L_z =: L > 0$. Let $N\in\mathbb{N}$ be given, and define the grid spacing $h := \frac{L}{N}$.  We will assume -- but only for simplicity of notation, ultimately -- that the mesh spacing in the $x$, $y$, and $z$-directions are the same. We define the following two uniform, infinite grids with grid spacing $h>0$:
	\[
E := \{ p_{i+\hf} \ |\ i\in {\mathbb{Z}}\}, \quad C := \{ p_i \ |\ i\in {\mathbb{Z}}\},
	\]
where $p_i = p(i) := (i-\hf)\cdot h$. Consider the following 3-D discrete $N^3$-periodic function spaces:
	\begin{eqnarray*}
	\begin{aligned}
{\mathcal C}_{\rm per} &:= \left\{\nu: C\times C
\times C\rightarrow {\mathbb{R}}\ \middle| \ \nu_{i,j,k} = \nu_{i+\alpha N,j+\beta N, k+\gamma N}, \ \forall \, i,j,k,\alpha,\beta,\gamma\in \mathbb{Z} \right\},
	\\
{\mathcal E}^{\rm x}_{\rm per} &:=\left\{\nu: E\times C\times C\rightarrow {\mathbb{R}}\ \middle| \ \nu_{i+\frac12,j,k}= \nu_{i+\frac12+\alpha N,j+\beta N, k+\gamma N}, \ \forall \, i,j,k,\alpha,\beta,\gamma\in \mathbb{Z}\right\} .
	\end{aligned}
	\end{eqnarray*}
Here we are using the identification $\nu_{i,j,k} = \nu(p_i,p_j,p_k)$, \emph{et cetera}. The spaces  ${\mathcal E}^{\rm y}_{\rm per}$ and ${\mathcal E}^{\rm z}_{\rm per}$ are analogously defined. The functions of ${\mathcal C}_{\rm per}$ are called {\emph{cell centered functions}}. The functions of ${\mathcal E}^{\rm x}_{\rm per}$, ${\mathcal E}^{\rm y}_{\rm per}$, and ${\mathcal E}^{\rm z}_{\rm per}$,  are called {\emph{east-west}},  {\emph{north-south}}, and {\emph{up-down face-centered functions}}, respectively.  We also define the mean zero space
	\[
\mathring{\mathcal C}_{\rm per}:=\left\{\nu\in {\mathcal C}_{\rm per} \ \middle| 0 = \overline{\nu} :=  \frac{h^3}{| \Omega|} \sum_{i,j,k=1}^m \nu_{i,j,k} \right\} .
	\]
We define $\vec{\mathcal{E}}_{\rm per} := {\mathcal E}^{\rm x}_{\rm per}\times {\mathcal E}^{\rm y}_{\rm per}\times {\mathcal E}^{\rm z}_{\rm per}$.

We now introduce the important difference and average operators on the spaces:
	\begin{eqnarray*}
&& A_x \nu_{i+\hf,j,k} := \frac{1}{2}\left(\nu_{i+1,j,k} + \nu_{i,j,k} \right), \quad D_x \nu_{i+\hf,j,k} := \frac{1}{h}\left(\nu_{i+1,j,k} - \nu_{i,j,k} \right),
	\\
&& A_y \nu_{i,j+\hf,k} := \frac{1}{2}\left(\nu_{i,j+1,k} + \nu_{i,j,k} \right), \quad D_y \nu_{i,j+\hf,k} := \frac{1}{h}\left(\nu_{i,j+1,k} - \nu_{i,j,k} \right) ,
	\\
&& A_z \nu_{i,j,k+\hf} := \frac{1}{2}\left(\nu_{i,j,k+1} + \nu_{i,j,k} \right), \quad D_z \nu_{i,j,k+\hf} := \frac{1}{h}\left(\nu_{i,j,k+1} - \nu_{i,j,k} \right) ,
	\end{eqnarray*}
with $A_x,\, D_x: {\mathcal C}_{\rm per}\rightarrow{\mathcal E}_{\rm per}^{\rm x}$, $A_y,\, D_y: {\mathcal C}_{\rm per}\rightarrow{\mathcal E}_{\rm per}^{\rm y}$, $A_z,\, D_z: {\mathcal C}_{\rm per}\rightarrow{\mathcal E}_{\rm per}^{\rm z}$.
Likewise,
	\begin{eqnarray*}
&& a_x \nu_{i, j, k} := \frac{1}{2}\left(\nu_{i+\hf, j, k} + \nu_{i-\hf, j, k} \right),	 \quad d_x \nu_{i, j, k} := \frac{1}{h}\left(\nu_{i+\hf, j, k} - \nu_{i-\hf, j, k} \right),
	\\
&& a_y \nu_{i,j, k} := \frac{1}{2}\left(\nu_{i,j+\hf, k} + \nu_{i,j-\hf, k} \right),	 \quad d_y \nu_{i,j, k} := \frac{1}{h}\left(\nu_{i,j+\hf, k} - \nu_{i,j-\hf, k} \right),
	\\
&& a_z \nu_{i,j,k} := \frac{1}{2}\left(\nu_{i, j,k+\hf} + \nu_{i, j, k-\hf} \right),	 \quad d_z \nu_{i,j, k} := \frac{1}{h}\left(\nu_{i, j,k+\hf} - \nu_{i, j,k-\hf} \right),
	\end{eqnarray*}
with $a_x,\, d_x : {\mathcal E}_{\rm per}^{\rm x}\rightarrow{\mathcal C}_{\rm per}$, $a_y,\, d_y : {\mathcal E}_{\rm per}^{\rm y}\rightarrow{\mathcal C}_{\rm per}$, and $a_z,\, d_z : {\mathcal E}_{\rm per}^{\rm z}\rightarrow{\mathcal C}_{\rm per}$.
The discrete gradient $\nabh:{\mathcal C}_{\rm per}\rightarrow \vec{\mathcal{E}}_{\rm per}$ is defined via
	\[
\nabh\nu_{i,j,k} =\left( D_x\nu_{i+\hf, j, k},  D_y\nu_{i, j+\hf, k},D_z\nu_{i, j, k+\hf}\right) ,
	\]
and the discrete divergence $\nabh\cdot :\vec{\mathcal{E}}_{\rm per} \rightarrow {\mathcal C}_{\rm per}$ is defined via
	\[
\nabh\cdot\vec{f}_{i,j,k} = d_x f^x_{i,j,k}	+ d_y f^y_{i,j,k} + d_z f^z_{i,j,k},
	\]
where $\vec{f} = (f^x,f^y,f^z)\in \vec{\mathcal{E}}_{\rm per}$. The standard 3-D discrete Laplacian, $\Delta_h : {\mathcal C}_{\rm per}\rightarrow{\mathcal C}_{\rm per}$, is given by
	\begin{align*}
\Delta_h \nu_{i,j,k} := & \nabla_{h}\cdot\left(\nabla_{h}\phi\right)_{i,j,k} =  d_x(D_x \nu)_{i,j,k} + d_y(D_y \nu)_{i,j,k}+d_z(D_z \nu)_{i,j,k}
	\\
= & \ \frac{1}{h^2}\left( \nu_{i+1,j,k}+\nu_{i-1,j,k}+\nu_{i,j+1,k}+\nu_{i,j-1,k}+\nu_{i,j,k+1}+\nu_{i,j,k-1} - 6\nu_{i,j,k}\right).
	\end{align*}
More generally, if $\mathcal{D}$ is a periodic \emph{scalar} function that is defined at all of the face center points and $\vec{f}\in\vec{\mathcal{E}}_{\rm per}$, then $\mathcal{D}\vec{f}\in\vec{\mathcal{E}}_{\rm per}$, assuming point-wise multiplication, and we may define
	\[
\nabla_h\cdot \big(\mathcal{D} \vec{f} \big)_{i,j,k} = d_x\left(\mathcal{D}f^x\right)_{i,j,k}  + d_y\left(\mathcal{D}f^y\right)_{i,j,k} + d_z\left(\mathcal{D}f^z\right)_{i,j,k} .
	\]
Specifically, if $\nu\in \mathcal{C}_{\rm per}$, then $\nabla_h \cdot\left(\mathcal{D} \nabla_h  \ \ \right):\mathcal{C}_{\rm per} \rightarrow \mathcal{C}_{\rm per}$ is defined point-wise via
	\[
\nabla_h\cdot \big(\mathcal{D} \nabla_h \nu \big)_{i,j,k} = d_x\left(\mathcal{D}D_x\nu\right)_{i,j,k}  + d_y\left(\mathcal{D} D_y\nu\right)_{i,j,k} + d_z\left(\mathcal{D}D_z\nu\right)_{i,j,k} .
	\]

Now we are ready to define the following grid inner products:
	\begin{equation*}
	\begin{aligned}
\ciptwo{\nu}{\xi} &:= h^3\sum_{i,j,k=1}^N  \nu_{i,j,k}\, \xi_{i,j,k},\quad \nu,\, \xi\in {\mathcal C}_{\rm per},\quad
& \eipx{\nu}{\xi} := \ciptwo{a_x(\nu\xi)}{1} ,\quad \nu,\, \xi\in{\mathcal E}^{\rm x}_{\rm per},
\\
\eipy{\nu}{\xi} &:= \ciptwo{a_y(\nu\xi)}{1} ,\quad \nu,\, \xi\in{\mathcal E}^{\rm y}_{\rm per},\quad
&\eipz{\nu}{\xi} := \ciptwo{a_z(\nu\xi)}{1} ,\quad \nu,\, \xi\in{\mathcal E}^{\rm z}_{\rm per}.
	\end{aligned}
	\end{equation*}	
	\[
\eipvec{\vec{f}_1}{\vec{f}_2} : = \eipx{f_1^x}{f_2^x}	+ \eipy{f_1^y}{f_2^y} + \eipz{f_1^z}{f_2^z}, \quad \vec{f}_i = (f_i^x,f_i^y,f_i^z) \in \vec{\mathcal{E}}_{\rm per}, \ i = 1,2.
	\]
	
We define the following norms for cell-centered functions. If $\nu\in {\mathcal C}_{\rm per}$, then $\nrm{\nu}_2^2 := \ciptwo{\nu}{\nu}$; $\nrm{\nu}_p^p := \ciptwo{|\nu|^p}{1}$, for $1\le p< \infty$, and $\nrm{\nu}_\infty := \max_{1\le i,j,k\le N}\left|\nu_{i,j,k}\right|$.
We define  norms of the gradient as follows: for $\nu\in{\mathcal C}_{\rm per}$,
	\[
\nrm{ \nabla_h \nu}_2^2 : = \eipvec{\nabh \nu }{ \nabh \nu } = \eipx{D_x\nu}{D_x\nu} + \eipy{D_y\nu}{D_y\nu} +\eipz{D_z\nu}{D_z\nu},
	\]
and, more generally, for $1\le p<\infty$,
	\begin{equation}
\nrm{\nabla_h \nu}_p := \left( \eipx{|D_x\nu|^p}{1} + \eipy{|D_y\nu|^p}{1} +\eipz{|D_z\nu|^p}{1}\right)^{\frac1p} .
	\end{equation}
Higher order norms can be defined. For example,
	\[
\nrm{\nu}_{H_h^1}^2 : =  \nrm{\nu}_2^2+ \nrm{ \nabla_h \nu}_2^2, \quad \nrm{\nu}_{H_h^2}^2 : =  \nrm{\nu}_{H_h^1}^2  + \nrm{ \Delta_h \nu}_2^2.
	\]

	\begin{lm}
	\label{lemma1}
Let $\mathcal{D}$ be an arbitrary periodic, scalar function defined on all of the face center points. For any $\psi, \nu \in {\mathcal C}_{\rm per}$ and any $\vec{f}\in\vec{\mathcal{E}}_{\rm per}$, the following summation by parts formulas are valid:
	\begin{equation}
\ciptwo{\psi}{\nabla_h\cdot\vec{f}} = - \eipvec{\nabla_h \psi}{ \vec{f}}, \quad \ciptwo{\psi}{\nabla_h\cdot \left(\mathcal{D}\nabla_h\nu\right)} = - \eipvec{\nabla_h \psi }{ \mathcal{D}\nabla_h\nu} .
	\label{lemma 1-0}
	\end{equation}
	\end{lm}

To facilitate the convergence analysis, we need to introduce a discrete analogue of the space $H_{per}^{-1}\left(\Omega\right)$, as outlined in~\cite{wang11a}. Suppose that $\mathcal{D}$ is a positive, periodic scalar function defined at all of the face center points. For any $\phi\in{\mathcal C}_{\rm per}$, there exists a unique $\psi\in\mathring{\mathcal C}_{\rm per}$ that solves
	\begin{eqnarray}
\mathcal{L}_{\mathcal{D}}(\psi):= - \nabla_h \cdot\left(\mathcal{D}\nabla_h \psi\right) = \phi - \overline{\phi} ,
	\end{eqnarray}
where, recall, $\overline{\phi} := |\Omega|^{-1}\ciptwo{\phi}{1}$. We equip this space with a bilinear form: for any $\phi_1,\, \phi_2\in \mathring{\mathcal C}_{\rm per}$, define
	\begin{equation}
\cipgen{ \phi_1 }{ \phi_2 }{\mathcal{L}_{\mathcal{D}}^{-1}} := \eipvec{\mathcal{D}\nabla_h \psi_1 }{ \nabla_h \psi_2 },
	\end{equation}
where $\psi_i\in\mathring{\mathcal C}_{\rm per}$ is the unique solution to
	\begin{equation}
\mathcal{L}_{\mathcal{D}}(\psi_i):= - \nabla_h \cdot\left(\mathcal{D}\nabla_h \psi_i\right)  = \phi_i, \quad i = 1, 2.
	\end{equation}
The following identity~\cite{wang11a} is easy to prove via summation-by-parts:
	\begin{equation}
\cipgen{\phi_1 }{ \phi_2 }{\mathcal{L}_{\mathcal{D}}^{-1}} = \ciptwo{\phi_1}{ \mathcal{L}_{\mathcal{D}}^{-1} (\phi_2) } = \ciptwo{ \mathcal{L}_{\mathcal{D}}^{-1} (\phi_1) }{\phi_2 },
	\end{equation}
and since $\mathcal{L}_{\mathcal{D}}$ is symmetric positive definite, $\cipgen{ \ \cdot \ }{\ \cdot \ }{\mathcal{L}_{\mathcal{D}}^{-1}}$ is an inner product on $\mathring{\mathcal C}_{\rm per}$~\cite{wang11a}. When $\mathcal{D}\equiv 1$, we drop the subscript and write $\mathcal{L}_{1} = \mathcal{L}$, and in this case we usually write $\cipgen{ \ \cdot \ }{\ \cdot \ }{\mathcal{L}_{\mathcal{D}}^{-1}} =: \cipgen{ \ \cdot \ }{\ \cdot \ }{-1,h}$. In the gerneral setting, the norm associated to this inner product is denoted $\nrm{\phi}_{\mathcal{L}_{\mathcal{D}}^{-1}} := \sqrt{\cipgen{\phi }{ \phi }{\mathcal{L}_{\mathcal{D}}^{-1}}}$, for all $\phi \in \mathring{\mathcal C}_{\rm per}$, but, if $\mathcal{D}\equiv 1$, we write $\nrm{\, \cdot \, }_{\mathcal{L}_{\mathcal{D}}^{-1}} =: \nrm{\, \cdot \, }_{-1,h}$.

\subsection{The first order numerical scheme and the main theoretical results}

We follow the idea of convexity splitting and consider the following semi-implicit, fully discrete schemes: given $\phi^n\in \mathcal{C}_{\rm per}$, find $\phi^{n+1},\mu^{n+1}\in  \mathcal{C}_{\rm per}$, such that
	\begin{align}
\frac{\phi^{n+1} - \phi^n}{\dt} = & \ - \hat{\cal M}^n  \mu^{n+1}   \quad  \mbox{(AC equation)} ,
	\label{scheme-AC_LOG-1}		
	\\
\frac{\phi^{n+1} - \phi^n}{\dt}  = & \ \nabla_h \cdot ( \check{\cal M}^n \nabla_h \mu^{n+1} )   \quad \mbox{(CH equation)} ,
	\label{scheme-CH_LOG-1}
	\end{align}
where
	\begin{equation}
\mu^{n+1} = \ln (1+\phi^{n+1}) - \ln (1-\phi^{n+1}) - \theta_0 \phi^n - \varepsilon^2 \Delta_h \phi^{n+1} .
      \label{scheme-mu-0}
	\end{equation}
The mobility approximations are defined as follows: for the Allen-Cahn approximation, $\hat{\cal M}^n = {\cal M}(\phi^n) \in \mathcal{C}_{\rm per}$, quite simply. For the Cahn-Hilliard approximation, we require that $\check{\cal M}^n$ is defined at all of the face center points. This is accomplished via
	\begin{align}  \label{scheme-CH_LOG-mobility-1}
\check{\cal M}_{i+\hf,j,k}^n = & \ {\cal M}(A_x \phi^n_{i+\hf,j,k}),	 \ \check{\cal M}_{i,j+\hf,k}^n = {\cal M}(A_y \phi^n_{i,j+\hf,k}),
	\\
\check{\cal M}_{i,j,k+\hf}^n = & \ {\cal M}(A_z \phi^n_{i,j,k+\hf}).
	\end{align}

Of course, a point-wise bound for the grid function $\phi^{n+1}$, namely, $-1 < \phi^{n+1}_{i,j,k} < 1$, is needed so that the numerical scheme is well-defined. The main theoretical results are stated below, which guarantee that there exist unique numerical solutions for \eqref{scheme-AC_LOG-1} and \eqref{scheme-CH_LOG-1}, so that the given bound is satisfied. In the first part, we assume that ${\cal M} (\phi) \equiv 1$; the non-constant mobility case will be analyzed in a later section.

For the Allen-Cahn equation, we have
	\begin{thm}
	\label{AC-positivity}
Assume that ${\cal M} (\phi) \equiv 1$. Given $\phi^n\in\mathcal{C}_{\rm per}$, with $\nrm{\phi^n}_\infty \le M$, for some $M >0$, there exists a unique solution $\phi^{n+1} \in \mathcal{C}_{\rm per}$ to \eqref{scheme-AC_LOG-1}, with $\nrm{\phi^{n+1}}_\infty < 1$. Moreover, if the initial data satisfy $\nrm{\phi^0}_\infty \le 1 - \delta_0$, there exists $\delta^\star\in (0,1)$,  which depends upon $\delta_0$ but is independent of $\varepsilon$ and $n$, so that $\nrm{\phi^n}_\infty \le 1 - \delta^\star$, $\forall n \in \mathbb{N}$.
	\end{thm}

For the Cahn-Hilliard equation, we have
	\begin{thm}
	\label{CH-positivity}
Assume that ${\cal M} (\phi) \equiv 1$. Given $\phi^n\in\mathcal{C}_{\rm per}$, with $\nrm{\phi^n}_\infty \le M$, for some $M >0$,  and $\left|\overline{\phi^n}\right| < 1$, there exists a unique solution $\phi^{n+1}\in\mathcal{C}_{\rm per}$ to \eqref{scheme-CH_LOG-1}, with $\phi^{n+1}-\overline{\phi^n}\in\mathring{\mathcal{C}}_{\rm per}$ and  $\nrm{\phi^{n+1}}_\infty < 1$.
	\end{thm}

\section{Theoretical justification of the positivity-preserving properties}
	\label{sec: proof}

	\subsection{Proof of Theorem~\ref{AC-positivity}}
		
The analysis for the approximation to the Allen-Cahn equation is given first.

	\begin{proof}
We observe that, the numerical solution of \eqref{scheme-AC_LOG-1} is equivalent to the minimization of the discrete energy functional
	\begin{eqnarray}
\mathcal{J}^n(\phi) &:=& \frac{1}{2 \dt} \| \phi - \phi^n \|_2^2 + \ciptwo{ 1+ \phi }{ \ln (1+\phi)}  + \ciptwo{ 1-\phi }{  \ln (1-\phi) }
	\nonumber
	\\
  &&
+ \frac{\varepsilon^2}{2} \| \nabla_h \phi \|_2^2  - \theta_0  \ciptwo{\phi }{ \phi^n },
	\label{AC_LOG-positive-1}
	\end{eqnarray}
over the compact, convex admissible set $A_h = \left\{ \phi \in \mathcal{C}_{\rm per} \ \middle| \ \nrm{\phi}_\infty \le 1 \right\} \subset \mathbb{R}^{N^3}$. We observe that $\mathcal{J}^n$ is a strictly convex function over this domain.  We wish to prove that there exists a minimizer of $\mathcal{J}^n$ at an interior point of $A_h$. To this end, consider the following closed domain: for a given $\delta\in (0,\nicefrac{1}{2})$,
	\begin{equation}
A_{h,\delta} := \left\{ \phi \in \mathcal{C}_{\rm per} \ \middle| \  \nrm{\phi}_\infty \le 1 - \delta \right\} \subset A_h .
	\label{AC_LOG-positive-2}
	\end{equation}
Since $A_{h,\delta}$ is a compact and convex set in $\mathbb{R}^{N^3}$, there exists a (not necessarily unique) minimizer of $\mathcal{J}^n$ over $A_{h,\delta}$. The key point of our positivity analysis is that such a minimizer could not occur on the boundary of $A_{h,\delta}$, if $\delta$ is small enough.

Assume a minimizer of $\mathcal{J}^n$ over $A_{h,\delta}$, denote it by $\phi^\star$, occurs at a boundary point. There is at least one grid point $\vec{\alpha}_0 = (i_0,j_0,k_0)$ such that $|\phi^\star_{\vec{\alpha}_0}| = 1 -\delta$. First, let us assume, that $\phi^\star_{\vec{\alpha}_0} = \delta - 1$, so that the grid function $\phi^\star$ has a global minimum at $\vec{\alpha}_0$.  Since $\mathcal{J}^n$ is smooth over $A_{h,\delta}$, for all $\psi\in \mathcal{C}_{\rm per}$, the directional derivative is
	\[
d_s \mathcal{J}^n(\phi^\star+s\psi)|_{s=0} = \ciptwo{ \frac{\phi^\star-\phi^n}{\Delta t}+\ln (1+\phi^\star) - \ln (1-\phi^\star) - \theta_0 \phi^n - \varepsilon^2 \Delta_h \phi^\star}{\psi}.
	\]
If the direction grid function is of the form $\psi_{i,j,k} = \delta_{i,i_0}\delta_{j,j_0}\delta_{k,k_0}$, where $\delta_{k,\ell}$ denotes the usual Kronecker delta function,
	\begin{equation}
\frac{1}{h^3} d_s \mathcal{J}^n(\phi^\star+s\psi)|_{s=0} =  \ln \delta - \ln (2-\delta) - \theta_0 \phi^n_{\vec{\alpha}_0} - \varepsilon^2 \Delta_h \phi^\star_{\vec{\alpha}_0}  + \frac{\delta -1 - \phi^n_{\vec{\alpha}_0}}{\dt} . \
	\label{AC_LOG-positive-3}
	\end{equation}
Since $\phi^\star$ has a minimum at the grid point $\vec{\alpha}_0 = (i_0, j_0, k_0)$,  it follows that
	\begin{equation}
\phi^\star_{\vec{\alpha}_0} = -1 + \delta \le \phi^\star_{i,j,k}, \quad \forall \ (i,j,k) \ne \vec{\alpha}_0, \quad \mbox{and} \quad \Delta_h \phi^\star_{\vec{\alpha}_0} \ge 0 .
	\label{AC_LOG-positive-4}
	\end{equation}
The bound $\nrm{\phi^n}_\infty \le M$ and the fact that $\delta\in (0,\nicefrac{1}{2})$ imply that
	\begin{equation}
\delta -1 - \phi^n_{\vec{\alpha}_0} \le \delta -1  + M < M-\nicefrac{1}{2}.
	\label{AC_LOG-positive-5}
	\end{equation}
Define the parameters
	\[
\beta_0  := 2 \left( 1 + \exp \left\{ \theta_0 M + \frac{M - \hf}{\dt}  \right\} \right)^{-1}, \quad \beta := \min(\nicefrac{1}{2},\beta_0).
	\]
If $\delta \in (0,\beta)$, then
	\begin{equation}
\ln \delta - \ln (2-\delta) - \theta_0 \phi^n_{\vec{\alpha}_0} + \frac{ \delta - 1 - \phi^n_{\vec{\alpha}_0}}{\dt} < 0 .
	\label{AC_LOG-positive-6}
	\end{equation}
Using the estimates \eqref{AC_LOG-positive-4} -- \eqref{AC_LOG-positive-6} in \eqref{AC_LOG-positive-3} reveals that, provided $0 < \delta < \beta$,
	\begin{equation}
\frac{1}{h^3} d_s \mathcal{J}^n(\phi^\star+s\psi)|_{s=0} < 0.
	\label{AC_LOG-positive-7}
	\end{equation}
This yields a contradiction that $\mathcal{J}^n$ takes a global minimum at $\phi^\star$ over $A_{h,\delta}$, because the directional derivative at this boundary point is \emph{negative} in a direction pointing into the interior of $A_{h,\delta}$. In other words, going in the direction of $\psi$, we are certain to find an interior point $\phi^\star+s\psi$, provided $s>0$ is sufficiently small, such that $\mathcal{J}^n(\phi^\star+s\psi) < \mathcal{J}^n(\phi^\star)$.

  Using quite similar arguments, if $\phi^\star_{\vec{\alpha}_0}=1-\delta$, and $\delta\in (0,\beta)$, we would find that 	
	\begin{equation}
\frac{1}{h^3} d_s \mathcal{J}^n(\phi^\star+s\psi)|_{s=0} > 0.
	\label{AC_LOG-positive-7-b}
	\end{equation}

A combination of these two facts shows that the global minimum of $\mathcal{J}^n$ over $A_{h,\delta}$ could only possibly occur at an interior point, when $\delta\in (0,\beta)$. We conclude that there must be a solution $\phi \in \left(A_{h,\delta}\right)^{\mathrm{o}}$, the interior region of $A_{h,\delta}$, so that for all $\psi\in\mathcal{C}_{\rm per}$,
	\begin{equation}
0 =  d_s \mathcal{J}^n(\phi+s\psi)|_{s=0} .
	\label{AC_LOG-positive-8}
	\end{equation}
which is equivalent to the numerical solution of \eqref{scheme-AC_LOG-1}, provided $\delta\in (0,\beta)$. The existence of a ``positive" numerical solution is, therefore, established.  In addition, since $\mathcal{J}^n$ is a strictly convex function over $A_h$, the uniqueness analysis for this numerical solution is straightforward.

For the second part of this theorem, let us make the \emph{a priori} assumption that, for some $\delta_0\in (0,1)$,  $\nrm{\phi^0}_\infty = 1-\delta_0$. Furthermore, choose $\delta_1 \in (0, 1)$ so that
	\[
\delta_1 < \frac{2}{ \exp(\theta_0 +1)}.
	\]
Define $\delta^\star = \min(\delta_0,\delta_1)$, and consider the space $A_{h,\delta^\star}$. Suppose that $\phi^{1,\star}$ is the minimizer of $J^0$ over $A_{h,\delta^\star}$. If we use an analysis similar to that of the first part, we can show that, if $\phi^{1,\star}$ is a boundary point of $A_{h,\delta^\star}$, we obtain a contradiction. Specifically, if at $\vec{\alpha}_0 = (i_0,j_0,k_0)$, $\phi^{\star,1}_{\vec{\alpha}_0} = \delta^\star - 1$ (a minimum point), then we find
	\begin{align*}
0 = & \  \ln \delta^\star - \ln (2-\delta^\star) - \theta_0 \phi^0_{\vec{\alpha}_0} - \varepsilon^2 \Delta_h \phi^{1,\star}_{\vec{\alpha}_0}  + \frac{\delta^\star -1 - \phi^0_{\vec{\alpha}_0}}{\dt}	
	\\
\le & \ \ln\delta^\star - \ln (2-\delta^\star) + \theta_0 < 0 .
	\end{align*}
Similarly, if  at $\vec{\alpha}_0 = (i_0,j_0,k_0)$, $\phi^{\star,1}_{\vec{\alpha}_0} =  1-\delta^\star$ (a maximum point), then we likewise discover that $0>0$. This implies, ultimately, that the minimizer $\phi^1\in A_h$ of $J^0$ satisfies the bound
	\[
\nrm{\phi^1}_\infty < 1- \delta^\star .
	\]
Clearly, $\delta^\star$ only depends on $\delta_0$ and $\theta_0$; it is independent of $\varepsilon$. This argument can be continued inductively, and we can conclude that, for any $n\in\mathbb{N}$,
	\[
\nrm{\phi^n}_\infty < 1- \delta^\star .
	\]
The proof of Theorem~\ref{AC-positivity} is complete.
	\end{proof}

	\subsection{Proof of Theorem~\ref{CH-positivity}}

If solutions to the Cahn-Hilliard scheme \eqref{scheme-CH_LOG-1} exist, it is clear that, for any $n\in\mathbb{N}$,
	\[
\overline{\phi}_0 :=  |\Omega|^{-1}\ciptwo{\phi^0}{1} = |\Omega|^{-1}\ciptwo{\phi^1}{1} = \cdots =  |\Omega|^{-1}\ciptwo{\phi^n}{1} = \overline{\phi}_n,
	\]
with $|\overline{\phi}_n |< 1$.  Thus we expect $\langle \phi^n - \overline{\phi}_0 , 1 \rangle_\Omega =0$. For the proof of Theorem~\ref{CH-positivity}, we need the following technical lemma:

	\begin{lm}
	\label{CH-positivity-Lem-0}
Suppose that $\phi_1$, $\phi_2 \in \mathcal{C}_{\rm per}$, with $\ciptwo{\phi_1 - \phi_2}{1} = 0$, that is, $\phi_1 - \phi_2\in \mathring{\mathcal{C}}_{\rm per}$, and assume that $\nrm{\phi_1}_\infty < 1$, $\nrm{\phi_2}_\infty \le M$. Then,   we have the following estimate:
	\begin{equation}
\nrm{\mathcal{L}^{-1} (\phi_1 - \phi_2)}_\infty \le C_1 ,
	\label{CH_LOG-Lem-0}
	\end{equation}
where $C_1>0$ depends only upon $M$ and $\Omega$. In particular, $C_1$ is independent of the mesh spacing $h$.
	\end{lm}

	\begin{proof}
Define $\psi :=  \phi_1 - \phi_2\in\mathring{\mathcal{C}}_{\rm per}$.  Thus $\nrm{\psi}_{\infty} < M+1$. This fact implies that
	\begin{equation}
\nrm{\psi}_2 = \nrm{\phi_1-\phi_2}_2 \le (M+1) | \Omega |^{1/2} .
	\label{CH_LOG-Lem-1}
	\end{equation}
Meanwhile, we denote $v = \mathcal{L}^{-1}(\psi)\in\mathring{\mathcal{C}}_{\rm per}$, so that $\mathcal{L}(v) = \psi$ with $\overline{v}=0$. Suppose that $N$ is odd, for simplicity, and $N=2K+1$. (The even case is handled in a very simliar manner.)  Since $v\in \mathcal{C}_{\rm per}$ it has the discrete Fourier representation of the form
	\begin{equation}
v_{i,j,k} = \sum^{K}_{\ell,m,n=-K}
\hat{v}^N_{\ell,m,n} {\rm e}^{2 \pi i ( \ell p_{i} + m p_{j} + n p_{k} )/ L } ,
	\label{def:Fourier-1}
	\end{equation}
where $p_{i} = (i-\hf)\cdot h$ and $\hat{v}^N_{\ell,m,n}$ are the discrete Fourier coefficients given by the discrete Fourier transform (DFT):
	\[
\hat{v}^N_{i,j,k} := \frac{h^3}{L^3} \sum^{K}_{\ell,m,n=-K}
v_{\ell,m,n} {\rm e}^{-2 \pi i ( \ell p_{i} + m p_{j} + n p_{k} )/ L } .
	\]
Since $v\in\mathring{\mathcal{C}}_{\rm per}$, $\hat{v}_{0,0,0}^N = 0$. We define the  Fourier interpolant of the grid function $v$ as
	\begin{equation}
	\label{def:extension}
\mathsf{v}(x,y,z) := \sum^{K}_{\ell,m,n=-K} \hat{v}^N_{\ell,m,n} {\rm e}^{2 \pi i ( \ell x + m y + n z )/ L } , \quad x,y,z\in\mathbb{R} ,
	\end{equation}
and observe that $\mathsf{v}\in C_{\rm per}^\infty(\Omega)$. Parseval's identity (at both the discrete and continuous levels) implies that
	\begin{equation}
\| v \|_2^2 = L^3 \sum^{K}_{\ell,m,n=-K}
|\hat{v}^N_{\ell,m,n}|^2 = \nrm{\mathsf{v}}_{L^2(\Omega)}^2 .
	\end{equation}
For the comparison between the discrete and continuous Laplacians, we start with the following Fourier expansions:
	\begin{eqnarray}
\Delta_h^x v_{i,j,k} &:=& \frac{v_{i+1,j,k}-2 v_{i,j,k}+v_{i-1,j,k}}{h^2}
	\nonumber
	\\
&=& \sum^{K}_{\ell,m,n=-K} \mu_{\ell} \hat{v}^N_{\ell,m,n}  {\rm e}^{2 \pi i  ( \ell x_{i} + m y_{j} + n z_{k} )/ L }  ,
	\\
  \partial_x^2 \mathsf{v} (x,y,z) &=& \sum^{K}_{\ell,m,n=-K}
  \nu_{\ell} \hat{v}^N_{\ell,m,n} {\rm e}^{2 \pi i
  ( \ell x + m y + n z )/ L } ,
	\end{eqnarray}
with
	\begin{equation}
\mu_{\ell} = -\frac{4\sin^2 {\frac{\ell\pi h}{L}}}{h^2}, \quad
\nu_{\ell} = -\frac{4 \ell^2 \pi^2}{L^2}.
\end{equation}
In turn, an application of Parseval's identity yields
\begin{eqnarray}
\nrm{\Delta_h^x v}^2_2 = L^3\sum^{K}_{\ell,m,n=-K}|\mu_{\ell}|^2|\hat{v}^N_{\ell,m,n}|^2, \\
\nrm{\partial_x^2 \mathsf{v}}^2_{L^2} =
 L^3\sum^{K}_{\ell,m,n=-K}|\nu_{\ell}|^2|
  \hat{v}^N_{\ell,m,n}|^2.
\end{eqnarray}
The comparison of Fourier eigenvalues shows that
\begin{equation}
\frac{4}{\pi^2} |\nu_{\ell}| \le |\mu_{\ell}| \le |\nu_{\ell}|,
\quad \mbox{for}  \quad -K \le \ell \le K .
\end{equation}
This indicates that
\begin{equation}
\frac{4}{\pi^2} \nrm{\partial_x^2 \mathsf{v}}_{L^2} \le \nrm{\Delta_h^x v}_2 \le \nrm{\partial_x^2 \mathsf{v}}_{L^2}.
\end{equation}
Similar estimates can be derived to reveal that
	\begin{equation}
  \frac{4}{\pi^2}\|\Delta \mathsf{v}\|_{L^2} \le \|\Delta_{h} v \|_2 = \| \psi  \|_2 \le \|\Delta \mathsf{v}\|_{L^2} ,
  	\label{CH_LOG-Lem-2}
	\end{equation}
which in turn yields that $\|\Delta \mathsf{v}\|_{L^2} \le \frac{(M+1) \pi^2 | \Omega |^{1/2}}{4}$.

Meanwhile, the following identity is obvious:
	\begin{equation}
\int_\Omega \mathsf{v} \, d {\bf x} = 0 ,  \quad \mbox{since} \quad \hat{v}_{0,0,0}^N = 0.
  \label{CH_LOG-Lem-3}
	\end{equation}
Subsequently, an application of elliptic regularity implies that
	\begin{equation}
\| \mathsf{v} \|_{H^2} \le C \left( \left| \int_\Omega \mathsf{v} \, d {\bf x} \right| +  \|\Delta \mathsf{v}\|_{L^2} \right)  \le C_0 (M+1) | \Omega |^{1/2} ,
	\label{CH_LOG-Lem-4}
	\end{equation}
for some constant $C_0>0$ that only depends upon $\Omega$. Since the grid function $v$ is the projection of the smooth function $\mathsf{v}$ into the cell-centered grid, the following discrete $\ell^\infty$ bound is clear:
	\begin{equation}
  \| v \|_\infty \le \| \mathsf{v} \|_{L^\infty} \le C \| \mathsf{v} \|_{H^2}  \le C_0 (M+1) | \Omega |^{1/2} ,
	\label{CH_LOG-Lem-5}
	\end{equation}
in which the 3-D Sobolev embedding has been used in the second step.  The proof of Lemma~\ref{CH-positivity-Lem-0} is completed by taking $C_1 := C_0 (M+1) | \Omega |^{1/2}$.
	\end{proof}

Now we proceed into the proof of Theorem~\ref{CH-positivity}.

	\begin{proof}
The numerical solution of \eqref{scheme-CH_LOG-1} is a minimizer of the following discrete energy functional:
	\begin{eqnarray}
\mathcal{J}^n (\phi) &:=& \frac{1}{2 \dt} \nrm{ \phi - \phi^n }_{-1,h}^2 + \ciptwo{ 1+ \phi }{ \ln (1+\phi) } + \ciptwo{ 1-\phi }{  \ln (1-\phi) }
	\nonumber
	\\
&&  + \frac{\varepsilon^2}{2} \| \nabla_h \phi \|_2^2  - \theta_0 \ciptwo{ \phi }{ \phi^n },
	\label{CH_LOG-positive-1}
	\end{eqnarray}
over the admissible set
	\[
A_h := \left\{ \phi \in \mathcal{C}_{\rm per} \ \middle| \  \nrm{\phi}_\infty \le 1,  \quad  \ciptwo{\phi-\overline{\phi}_0}{1}=0 \right\} \subset \mathbb{R}^{N^3}.
	\]
Observe that $\mathcal{J}^n$ is a strictly convex function over this domain.

To facilitate the analysis below, we transform the minimization problem into an equivalent one. Consider the functional
	\begin{eqnarray}
\mathcal{F}^n (\varphi) &:=& \mathcal{J}^n (\varphi + \overline{\phi}_0)
	\\
&=&  \frac{1}{2 \dt} \nrm{ \varphi + \overline{\phi}_0 - \phi^n}_{-1,h}^2 + \ciptwo{ 1+ \varphi + \overline{\phi}_0 }{ \ln ( 1 + \varphi + \overline{\phi}_0 )}
	\nonumber
	\\
&& + \ciptwo{1- \varphi - \overline{\phi}_0 }{  \ln (1-\varphi - \overline{\phi}_0)  } + \frac{\varepsilon^2}{2} \nrm{ \nabla_h \varphi }_2^2  - \theta_0 \ciptwo{ \varphi + \overline{\phi}_0 }{ \phi^n } ,
	\label{CH_LOG-positive-1-b}
	\end{eqnarray}
defined on the set
	\[
\mathring{A}_h := \left\{ \varphi \in \mathring{\mathcal{C}}_{\rm per} \ \middle| \  -1-\overline{\phi}_0 \le \varphi \le 1-\overline{\phi}_0  \right\} \subset \mathbb{R}^{N^3}.
	\]
If $\varphi\in \mathring{A}_h$ minimizes $\mathcal{F}^n$, then $\phi := \varphi + \overline{\phi}_0\in A_h$ minimizes $\mathcal{J}^n$, and \emph{vice versa}. Next, we prove that there exists a minimizer of $\mathcal{F}^n$ over the domain $\mathring{A}_h$. Similar to our previous arguments, we consider the following closed domain: for $\delta\in (0,\hf)$,
	\begin{equation}
\mathring{A}_{h,\delta} := \left\{ \varphi \in \mathring{\mathcal{C}}_{\rm per} \ \middle| \  \delta -1-\overline{\phi}_0 \le \varphi \le 1-\delta-\overline{\phi}_0  \right\} \subset \mathbb{R}^{N^3}.
	\label{CH_LOG-positive-2}
	\end{equation}
Since $\mathring{A}_{h,\delta}$ is a bounded, compact, and convex set in the subspace $\mathring{\mathcal{C}}_{\rm per}$, there exists a (not necessarily  unique) minimizer of $\mathcal{F}^n$ over $\mathring{A}_{h, \delta}$. The key point of the positivity analysis is that such a minimizer could not occur on the boundary of $\mathring{A}_{h,\delta}$, if $\delta$ is sufficiently small. To be more explicit, by the boundary of $\mathring{A}_{h, \delta}$, we mean the locus of points $\psi\in\mathring{A}_{h, \delta}$ such that $\nrm{\psi+\overline{\phi}_0}_\infty = 1-\delta$, precisely.

To get a contradiction, suppose that the minimizer of $\mathcal{F}^n$, call it $\varphi^\star$ occurs at a boundary point of $\mathring{A}_{h,\delta}$.  There is at least one grid point $\vec{\alpha}_0 = (i_0,j_0,k_0)$ such that $|\varphi^\star_{\vec{\alpha}_0}+\overline{\phi}_0| = 1 -\delta$. First, let us assume, that $\varphi^\star_{\vec{\alpha}_0}+\overline{\phi}_0 = \delta - 1$, so that the grid function $\varphi^\star$ has a global minimum at $\vec{\alpha}_0$. Suppose that $\vec{\alpha}_1 = (i_1,j_1,k_1)$ is a grid point at which $\varphi^\star$ achieves its maximum. By the fact that $\overline{\varphi^\star} = 0$, it is obvious that
	\[
1-\delta \ge \varphi^\star_{\vec{\alpha}_1}+\overline{\phi}_0 \ge \overline{\phi}_0.
	\]
Since $\mathcal{F}^n$ is smooth over $\mathring{A}_{h,\delta}$, for all $\psi\in \mathring{\mathcal{C}}_{\rm per}$, the directional derivative is
	\begin{align*}
 d_s \mathcal{F}^n(\varphi^\star+s\psi)|_{s=0} = & \  \ciptwo{ \ln (1+\varphi^\star+\overline{\phi}_0) - \ln (1-\varphi^\star-\overline{\phi}_0)}{\psi}
	\\
& -  \ciptwo{\theta_0 \phi^n + \varepsilon^2 \Delta_h \varphi^\star}{\psi}
	\\
& + \frac{1}{\Delta t}\ciptwo{(-\Delta_h)^{-1}\left(\varphi^\star-\phi^n+\overline{\phi}_0 \right)}{\psi}.
	\end{align*}
This time, let us pick the direction $\psi \in \mathring{\mathcal{C}}_{\rm per}$, such that
	\[
\psi_{i,j,k} = \delta_{i,i_0}\delta_{j,j_0}\delta_{k,k_0} - \delta_{i,i_1}\delta_{j,j_1}\delta_{k,k_1} .
	\]
Then the derivative may be expressed as
	\begin{eqnarray}
\frac{1}{h^3} d_s \mathcal{F}^n(\varphi^\star+s\psi)|_{s=0}  &=&  \ln (1+ \varphi^\star_{\vec{\alpha}_0}+\overline{\phi}_0) - \ln (1-\varphi^\star_{\vec{\alpha}_0}-\overline{\phi}_0)
	\nonumber
	\\
&& - \ln (1+ \varphi^\star_{\vec{\alpha}_1}+\overline{\phi}_0) + \ln (1-\varphi^\star_{\vec{\alpha}_1}-\overline{\phi}_0)
	\nonumber
	\\
&& - \theta_0 ( \phi^n_{\vec{\alpha}_0} - \phi^n_{\vec{\alpha}_1} ) - \varepsilon^2 ( \Delta_h \varphi^\star_{\vec{\alpha}_0} -  \Delta_h \varphi^\star_{\vec{\alpha}_1} )
	\nonumber
	\\
&& + \frac{1}{\dt}  (-\Delta_h)^{-1} ( \varphi^\star  - \phi^n +\overline{\phi}_0 )_{\vec{\alpha}_0}
	\nonumber
	\\
&& - \frac{1}{\dt}(-\Delta_h)^{-1} ( \varphi^\star - \phi^n + \overline{\phi}_0)_{\vec{\alpha}_1} .  \label{CH_LOG-positive-4}
	\end{eqnarray}
For simplicity, now let us write $\phi^\star := \varphi^\star +\overline{\phi}_0$. Since  $\phi^\star_{\vec{\alpha}_0} = -1 + \delta$ and $\phi^\star_{\vec{\alpha}_1} \ge \overline{\phi}_0$, we have
	\begin{equation}
\ln (1+ \phi^\star_{\vec{\alpha}_0})  - \ln (1-\phi^\star_{\vec{\alpha}_0}) - \ln (1+ \phi^\star_{\vec{\alpha}_1}) + \ln (1-\phi^\star_{\vec{\alpha}_1}) \le \ln \frac{\delta}{2 - \delta} - \ln \frac{1+\overline{\phi}_0}{1-\overline{\phi}_0} .
	\label{CH_LOG-positive-5}
	\end{equation}
Since $\phi^\star$ takes a  minimum at the grid point $\vec{\alpha}_0$, with $\phi^\star_{\vec{\alpha}_0} = -1 + \delta \le \phi^\star_{i,j,k}$, for any $(i,j,k)$, and a maximum at the grid point $\vec{\alpha}_1$, with $\phi^\star_{\vec{\alpha}_1} \ge \phi^\star_{i,j,k}$, for any $(i,j,k)$,
	\begin{equation}
\Delta_h \phi^\star_{\vec{\alpha}_0} \ge 0 ,  \quad \Delta_h \phi^\star_{\vec{\alpha}_1} \le 0 .
	\label{CH_LOG-positive-6}
	\end{equation}
For the numerical solution $\phi^n$ at the previous time step, the \emph{a priori} assumption $\nrm{\phi^n}_\infty \le M$ indicates that
	\begin{equation}
-2 M \le \phi^n_{\vec{\alpha}_0} - \phi^n_{\vec{\alpha}_1} \le 2M .
	\label{CH_LOG-positive-8}
	\end{equation}
For the last two terms appearing in (\ref{CH_LOG-positive-4}), we apply Lemma~\ref{CH-positivity-Lem-0} and obtain
	\begin{equation}
- 2 C_1 \le (-\Delta_h)^{-1} ( \phi^\star - \phi^n )_{\vec{\alpha}_0} - (-\Delta_h)^{-1} ( \phi^\star - \phi^n )_{\vec{\alpha}_1} \le  2 C_1 .
	\label{CH_LOG-positive-9}
	\end{equation}
Consequently,  a substitution of (\ref{CH_LOG-positive-5}) -- (\ref{CH_LOG-positive-9}) into (\ref{CH_LOG-positive-4}) yields the following bound on the directional derivative:
	\begin{equation}
\frac{1}{h^3} d_s \mathcal{F}^n(\varphi^\star+s\psi)|_{s=0} \le \ln \frac{\delta}{2 - \delta} - \ln \frac{1+\overline{\phi}_0}{1-\overline{\phi}_0}  + 2M \theta_0 + 2 C_1 \dt^{-1} .
	\label{CH_LOG-positive-10}
	\end{equation}
We denote $C_2 = 2M \theta_0 + 2 C_1 \dt^{-1}$. Note that $C_2$ is a constant for a fixed $\dt$, though it becomes singular as $\dt \to 0$. However, for any fixed $\dt$, we may choose $\delta\in(0,\hf)$ sufficiently small so that
	\begin{equation}
\ln \frac{\delta}{2 - \delta} - \ln \frac{1+\overline{\phi}_0}{1-\overline{\phi}_0} + C_2 < 0 .  \label{CH_LOG-positive-11}
	\end{equation}
This in turn shows that, provided $\delta$ satisfies (\ref{CH_LOG-positive-11}),
	\begin{equation}
\frac{1}{h^3} d_s \mathcal{F}^n(\varphi^\star+s\psi)|_{s=0} < 0 .
	\label{CH_LOG-positive-12}
	\end{equation}
As before, this contradicts the assumption that $\mathcal{F}^n$ has a minimum at $\varphi^\star$, since the directional derivative is negative in a direction pointing into the interior of $\mathring{A}_{h,\delta}$.

Using very similar arguments, we can also prove that the global minimum of $\mathcal{F}^n$ over $\mathring{A}_{h,\delta}$ could not occur at a boundary point $\varphi^\star$ such that  $\varphi^\star_{\vec{\alpha}_0}+\overline{\phi}_0 = 1-\delta$, for some $\vec{\alpha}_0$, so that the grid function $\varphi^\star$ has a global maximum at $\vec{\alpha}_0$. The details are left to interested readers.

A combination of these two facts shows that, the global minimum of $\mathcal{F}^n$ over $\mathring{A}_{h,\delta}$ could only possibly occur at interior point $\varphi\in (\mathring{A}_{h,\delta})^{\rm o}\subset (\mathring{A}_h)^{\rm o}$. We conclude that there must be a solution $\phi = \varphi+\overline{\phi}_0\in A_h$ that minimizes $\mathcal{J}^n$ over $A_h$, which is equivalent to the numerical solution of (\ref{scheme-CH_LOG-1}), (\ref{scheme-mu-0}). The existence of the numerical solution is established.

In addition, since $\mathcal{J}^n$ is a strictly convex function over $A_h$, the uniqueness analysis for this numerical solution is straightforward. The proof of Theorem~\ref{CH-positivity} is complete.
	\end{proof}

	\begin{rem}
The positivity-preserving analysis is based on a key fact that the singular nature of the logarithmic term around the values of $-1$ and 1 prevents the numerical solution reaching these singular values. As a result, the point-wise positivity for the logarithmic arguments could be derived as long as the numerical solution at the previous time step stays bounded between $-M$ and $M$ (even if $M >1$), and the initial average stays between $-1$ and 1. This is a modest improvement to the results in~\cite{elliott92a}, in which the authors constructed a cut-off energy functional to avoid the singularity.
	\end{rem}

	\begin{rem}
The proof of Theorem~\ref{AC-positivity} follows a standard maximum principle type argument; that is the key reason why we are able to obtain a uniform separation bound for the numerical solution (\ref{scheme-AC_LOG-1}):  $\nrm{\phi^n}_\infty \le 1 - \delta^\star$, if the initial data satisfy a similar condition. For the Cahh-Hilliard flow, such a uniform bound is not available for the corresponding numerical solution (\ref{scheme-CH_LOG-1}), (\ref{scheme-mu-0}) any more, since the maximum principle could not be directly applied to an $H^{-1}$ gradient flow. In addition, the mass conservation constraint has made the corresponding analysis more involved.
	\end{rem}

	\begin{rem}
For the Cahn-Hilliard flow, lack of maximum principle has been an essential mathematical challenge. To overcome this difficulty, we have to obtain the point-wise bound for the linear chemical potential part. With the help of the \emph{a priori} $\ell^\infty$ bound of the numerical solution we are investigating, an $O (\dt^{-1})$ estimate is derived for such a bound, which is contained in the form of $C_2$. Such a bound is a fixed constant for a fixed $\dt$, while it becomes singular as $\dt \to 0$.

Another key idea of this analysis should also be mentioned: although the nonlinear term contains a singular limit as $\phi$ approaches either $-1$ or 1, the convexity of this nonlinear potential has greatly aided in the positivity analysis. 
	\end{rem}

	\begin{rem}
In addition to the positivity-preserving property, the semi-implicit nature of our proposed scheme: implicit treatment for the logarithmic terms and the surface diffusion term, combined with an explicit treatment for the linear stretching/expansive term, ensures the unique solvability. In comparison, for the fully implicit scheme analyzed in~\cite{elliott92a}, the unique solvability is only available under a time step constraint: $\dt \le \frac{4 \varepsilon^2}{\theta_0^2}$. In fact, the existence of the positivity-preserving numerical solution could also be established for the fully implicit Euler scheme, using the same idea presented in this section. Only the uniqueness analysis of the numerical solution requires such a time step
constraint.
	\end{rem}

	\begin{rem}
For simplicity of presentation, we only analyze the finite difference scheme over a rectangular domain in this article. The idea of this positivity analysis could be similarly extended to the finite element and pseudo-spectral spatial approximations, as well as the case of a general domain. The details may be considered in the future works.
	\end{rem}

\subsection{The positivity preserving property in the non-constant mobility case}

In this subsection we look at the numerical scheme (\ref{scheme-CH_LOG-1}), (\ref{scheme-mu-0}) with a nonconstant mobility, but with the strict positivity assumption that ${\cal M} (x) \ge  {\cal M}_0 >0 $, for all $x\in[-1,1]$. Then, for any $\phi  \in\mathring{\mathcal C}_{\Omega}$, there exists a unique $\psi\in\mathring{\mathcal C}_{\Omega}$ that solves
	\begin{equation}
\mathcal{L}_{\check{\cal M}^n} (\psi):= - \nabla_h \cdot ( \check{\cal M}^n \nabla_h \psi ) = \phi .
	\label{CH-mobility-0}
	\end{equation}
In turn, the following norm may be introduced:
	\begin{equation}
\nrm{ \phi  }_{\mathcal{L}_{\check{\cal M}^n}^{-1} } = \sqrt{\ciptwo{\phi}{ \mathcal{L}_{\check{\cal M}^n}^{-1} (\phi) }} .
	\end{equation}
Similar to Lemma~\ref{CH-positivity-Lem-0}, the following estimate is needed in the positivity analysis.

	\begin{lm}
	\label{CH-mobility-positivity-Lem-0}
Suppose that $\phi_1$, $\phi_2 \in \mathcal{C}_{\rm per}$, with $\ciptwo{\phi_1 - \phi_2}{1} = 0$, that is, $\phi_1 - \phi_2\in \mathring{\mathcal{C}}_{\rm per}$, and assume that $\nrm{\phi_1}_\infty < 1$, $\nrm{\phi_2}_\infty \le M$. Then,   we have the following estimate:
	\begin{equation}
\nrm{\mathcal{L}_{\check{\cal M}^n}^{-1} (\phi_1 - \phi_2)}_\infty \le C_4 := C_3 \mathcal{M}_0^{-1} h^{-1/2} ,
	\label{CH_LOG-mobility-Lem-0}
	\end{equation}
where $C_3>0$ depends only upon $M$ and $\Omega$.
	\end{lm}

	\begin{proof}
Define $\psi :=\phi_1 - \phi_2$ and $v:= \mathcal{L}_{\check{\mathcal{M}}^n}^{-1}(\psi)$  Similar to the estimate~\eqref{CH_LOG-Lem-1}, we get
\begin{eqnarray}
  \| \psi \|_2 = \| \phi_1 - \phi_2 \|_2 \le (M+1) | \Omega |^{1/2} .  \label{CH_LOG-mobility-Lem-1}
\end{eqnarray}
To obtain a bound for $v\in\mathring{\mathcal{C}}_{\rm per}$, observe that, by summation-by-parts,
	\begin{equation}
\mathcal{M}_0\nrm{\nabla_h v}_2^2 \le \eipvec{\check{\cal M}^n \nabla_h v }{ \nabla_h v } = \ciptwo{\psi }{ v} \le \| \psi \|_2 \cdot \| v \|_2 \le C_{\rm P} \| \psi \|_2 \cdot \| \nabla_h v \|_2 ,
  	\label{CH_LOG-mobility-Lem-2}
	\end{equation}
in which the discrete Poincar\'e inequality,
	\[
\| \psi \|_2 \le C_{\rm P} \| \nabla_h \psi \|_2,  \quad \forall \psi \in \mathring{\mathcal{C}}_{\rm per},
	\]
has been applied in the last step.  Therefore
	\begin{equation}
\nrm{\nabla_h v}_2 \le C_{\rm P} \mathcal{M}_0^{-1} \nrm{ \psi }_2 .
	\label{CH_LOG-mobility-Lem-3}
	\end{equation}
Subsequently, an application of a 3-D inverse inequality, for $v\in \mathring{\mathcal{C}}_{\rm per}$, leads to
	\begin{eqnarray}
\nrm{ v }_\infty &\le& C_{\rm I} h^{-1/2} \nrm{ \nabla_h v}_2  \le C_{\rm I} h^{-1/2} C_{\rm P} \mathcal{M}_0^{-1} \nrm{ \psi }_2
	\nonumber
	\\
&\le & C_{\rm I} h^{-1/2} C_{\rm P} \mathcal{M}_0^{-1} (M+1) | \Omega |^{1/2} ,
	\label{CH_LOG-mobility-Lem-4}
	\end{eqnarray}
where the constant in the inverse inequality, $C_{\rm I}>0$, is independent of $h$. Therefore, \eqref{CH_LOG-mobility-Lem-0} is valid, with $C_3 :=  C_{\rm I}  C_{\rm P} (M+1) | \Omega |^{1/2}$. This completes the proof.
	\end{proof}

The positivity-preserving property of the numerical scheme (\ref{scheme-CH_LOG-1}), (\ref{scheme-mu-0}) for the non-constant mobility case is stated below.

	\begin{thm}
	\label{CH-mobility-positivity}
Assume that ${\cal M} (x) \ge  {\cal M}_0 >0 $, for all $x\in[-1,1]$. Given $\phi^n\in\mathcal{C}_{\rm per}$, with $\nrm{\phi^n} \le M$, for some $M >0$,  and $\left|\overline{\phi^n}\right| < 1$, there exists a unique solution $\phi^{n+1}\in\mathcal{C}_{\rm per}$ to \eqref{scheme-CH_LOG-1}, with $\phi^{n+1}-\overline{\phi^n}\in\mathring{\mathcal{C}}_{\rm per}$ and  $\nrm{\phi^{n+1}}_\infty < 1$.
	\end{thm}

	\begin{proof}
The proof of this theorem follows the same ideas as in  that of Theorem~\ref{CH-positivity}; we just provide a brief outline. Similar to~\eqref{CH_LOG-positive-1},  the numerical solution of (\ref{scheme-CH_LOG-1}) is equivalent to the minimization of the following discrete energy functional:
	\begin{eqnarray}
\mathcal{J}^n (\phi) &=& \frac{1}{2 \dt} \nrm{ \phi - \phi^n }_{ \mathcal{L}_{\check{\cal M}^n}^{-1}}^2 + \ciptwo{1+ \phi }{ \ln (1+\phi) } + \ciptwo{ 1-\phi }{\ln (1-\phi) }
	\nonumber
	\\
& &  + \frac{\varepsilon^2}{2} \nrm{ \nabla_h \phi }_2^2  - \theta_0 \ciptwo{ \phi }{ \phi^n } ,
	\label{CH_LOG-mobility-positive-1}
	\end{eqnarray}
over the admissible set
	\[
A_h := \left\{ \phi \in \mathcal{C}_{\rm per} \ \middle| \  \nrm{\phi}_\infty \le 1,  \  \ciptwo{\phi-\overline{\phi}_0}{1}=0 \right\} .
	\]
The equivalent minimization problem is similar to previous one: find a minimizer $\varphi\in \mathring{A}_h$  the functional
	\[
\mathcal{F}^n (\varphi) := \mathcal{J}^n (\varphi + \overline{\phi}_0), \quad \mbox{with} \quad
\mathring{A}_h := \left\{ \varphi \in \mathring{\mathcal{C}}_{\rm per} \ \middle| \  -1-\overline{\phi}_0 \le \varphi \le 1-\overline{\phi}_0  \right\} \subset \mathbb{R}^{N^3}.
	\]
There exists a (not necessarily  unique) minimizer of $\mathcal{F}^n$ over the restricted set $\mathring{A}_{h, \delta}$, defined in \eqref{CH_LOG-positive-2}, where $\delta\in(0,\hf)$.  To get a contradiction, suppose that the minimizer of $\mathcal{F}^n$, call it $\varphi^\star$, occurs at a boundary point of $\mathring{A}_{h,\delta}$.  There is at least one grid point $\vec{\alpha}_0 = (i_0,j_0,k_0)$ such that $|\varphi^\star_{\vec{\alpha}_0}+\overline{\phi}_0| = 1 -\delta$. As before, we first  assume  that $\varphi^\star_{\vec{\alpha}_0}+\overline{\phi}_0 = \delta - 1$, so that the grid function $\varphi^\star$ has a global minimum at $\vec{\alpha}_0$. Suppose that $\vec{\alpha}_1 = (i_1,j_1,k_1)$ is a grid point at which $\varphi^\star$ achieves its maximum.

The directional derivative, in the direction
	\[
\psi_{i,j,k} = \delta_{i,i_0}\delta_{j,j_0}\delta_{k,k_0} - \delta_{i,i_1}\delta_{j,j_1}\delta_{k,k_1} ,
	\]
satisfies
	\begin{eqnarray}
\frac{1}{h^3} d_s \mathcal{F}^n(\varphi^\star+s\psi)|_{s=0}  &=&  \ln (1+ \varphi^\star_{\vec{\alpha}_0}+\overline{\phi}_0) - \ln (1-\varphi^\star_{\vec{\alpha}_0}-\overline{\phi}_0)
	\nonumber
	\\
&& - \ln (1+ \varphi^\star_{\vec{\alpha}_1}+\overline{\phi}_0) + \ln (1-\varphi^\star_{\vec{\alpha}_1}-\overline{\phi}_0)
	\nonumber
	\\
&& - \theta_0 ( \phi^n_{\vec{\alpha}_0} - \phi^n_{\vec{\alpha}_1} ) - \varepsilon^2 ( \Delta_h \varphi^\star_{\vec{\alpha}_0} -  \Delta_h \varphi^\star_{\vec{\alpha}_1} )
	\nonumber
	\\
&& + \frac{1}{\dt}  \mathcal{L}_{\check{\mathcal{M}}^n}^{-1} ( \varphi^\star  - \phi^n +\overline{\phi}_0 )_{\vec{\alpha}_0}  - \frac{1}{\dt}\mathcal{L}_{\check{\mathcal{M}}^n}^{-1} ( \varphi^\star - \phi^n + \overline{\phi}_0)_{\vec{\alpha}_1} .
	\nonumber
	\\
&&
	\label{CH_LOG-mobility-positive-4}
	\end{eqnarray}
We now apply Lemma~\ref{CH-mobility-positivity-Lem-0} to obtain (keeping in mind that $\phi^* = \varphi^* + \bar{\phi}_0$)
	\begin{equation}
  - 2 C_4 \le   \mathcal{L}_{\check{\cal M}^n}^{-1} ( \phi^\star - \phi^n )_{\vec{\alpha}_0}  - (  \mathcal{L}_{\check{\cal M}^n}^{-1} ( \phi^\star - \phi^n )_{\vec{\alpha}_1} \le  2 C_4 .
	\label{CH_LOG-mobility-positive-9}
	\end{equation}
This, together with some other estimates, obtained as in the proof of Theorem~\ref{CH-positivity}, yields
	\begin{equation}
\frac{1}{h^3} d_s \mathcal{F}^n(\varphi^\star+s\psi)|_{s=0} \le \ln \frac{\delta}{2 - \delta} - \ln \frac{1+\overline{\phi}_0}{1-\overline{\phi}_0}  +  C_5.
	\label{CH_LOG-mobility-positive-10}
	\end{equation}
where $C_5:= 2M\theta_0 + 2 C_4 \dt^{-1}$. For $\delta\in(0,\hf)$ sufficiently small, the right hand side is strictly less than $0$. The rest of the analysis follows the proof of Theorem~\ref{CH-positivity}; the details are left to the interested readers.
	\end{proof}


	\begin{rem}
In the proof of Theorem~\ref{CH-mobility-positivity}, the point-wise positivity of the mobility, $\check{\cal M}^n \ge \mathcal{M}_0 > 0$, is assumed for the convenience of the analysis. However, at the PDE level, the CH flow with a degenerate mobility has been analyzed in~\cite{barrett99, elliott96b}. The numerical scheme for the degenerate mobility equation will also be considered in the authors' future works. In fact, our assumption could be relaxed to allow for certain mobilities satisfying ${\cal M} (\phi^n) > 0$ at a point-wise level; the technical details are left to interested readers. In particular, for the case of the standard symmetric degenerate mobility, $\mathcal{M} (\phi) = (1-\phi)(1+\phi)$, the PDE analyses for which were undertaken  by~\cite{cahn1996,elliott96b}, our analysis would go through, with the help of a subtle fact that $\mathcal{M} (\phi)$ only degenerates at $\phi=-1$ and 1, combined with the positivity-preserving result at the previous time step.
	\end{rem}

\section{Unconditional energy stability and uniform in time $H_h^1$ bound} \label{sec:energy stability}

The discrete energy is defined as
	\begin{eqnarray}
E_h (\phi) = \ciptwo{ 1+ \phi }{ \ln (1+\phi) } + \ciptwo{ 1-\phi }{  \ln (1-\phi) }  + \frac{\varepsilon^2}{2} \nrm{ \nabla_h \phi }_2^2  - \frac{\theta_0}{2} \nrm{ \phi }_2^2  .
	\label{CH-discrete energy}
	\end{eqnarray}
For the numerical scheme for the Cahn-Hilliard equation (\ref{scheme-CH_LOG-1}), (\ref{scheme-mu-0}), the existence and unique solvability (so that the numerical solution stays within $(-1,1)$ at a point-wise level) have been established in Theorem~\ref{CH-mobility-positivity}. Because the scheme uses a convex-concave decomposition, it is unconditionally energy stability. This result is stated in the following theorem, whose proof is omitted for the sake of brevity and also because it is standard:

	\begin{thm}
	\label{CH-mobility-energy stability}
For simplicity, suppose that $N=2K+1$, and let $\mathcal{P}_N:C_{\rm per}(\Omega)\to \mathcal{B}_K(\Omega)$ denote the Fourier projection operator, where $\mathcal{B}_K$ is space of $\Omega$-periodic (complex) trigonometric polynomials of degree up to and including $K$. By $\mathcal{P}_h:C_{\rm per}(\Omega)\to \mathcal{C}_{\rm per}$ denote the canonical grid projection operator. Suppose that  $\phi^0:= \mathcal{P}_h(\mathcal{P}_N\Phi)$, where $\Phi\in C^6_{\rm per}(\Omega)$ and $\nrm{\Phi}_{L^\infty}<1$. Then $(\Phi,1)_{L^2} = \ciptwo{\phi^0}{1}$, and, for any $\dt>0$, $h >0$, and $m\in\mathbb{N}$,
	\[
E_h(\phi^m) + \eipvec{\check{\mathcal{M}}^{m-1}\nabla_h\mu^m}{\nabla_h\mu^m} \le E_h(\phi^{m-1}),	
	\]
so that $E_h(\phi^m) \le E_h(\phi^0) \le  C_6$, with $C_6>0$ independent of $h$. Therefore, since $- \frac{\theta_0}{2} | \Omega |
  + \frac{\varepsilon^2}{2} \| \nabla_h \phi^m \|_2^2 \le E_h (\phi^m)$, we have
	\begin{equation}
\nrm{ \nabla_h \phi^m }_2 \le \sqrt{2 C_6 +  \theta_0 | \Omega | } \varepsilon^{-1} =:C_7 , \quad \forall m\in\mathbb{N} .
	\label{CH-mobility-H1-2-3}
	\end{equation}
	\end{thm}

	\begin{rem}
The unconditional energy stability of the proposed scheme (\ref{scheme-CH_LOG-1}), (\ref{scheme-mu-0}) follows from the convex-concave decomposition of the energy, an idea popularized in Eyre's work~\cite{eyre98}. The method has been applied to the phase field crystal (PFC) equation and the modified version~\cite{wang11a, wise09};  epitaxial thin film growth models~\cite{chen12, wang10a}; non-local gradient model~\cite{guan14a}; the Cahn-Hilliard model coupled with fluid flow~\cite{chen16, diegel15a, feng12, LiuY17, wise10}; \emph{et cetera}. Second order accurate energy stable schemes have also been reported in recent years, based on either a secant/Crank-Nicolson or BDF approach.  See, for example, \cite{baskaran13a, baskaran13b, chen14, diegel16, diegel17, guo16, han15, hu09, shen12, guan14b, yan17}. In particular, for the multi-component Cahn-Hilliard model, the related works could also be found in~\cite{barrett97, barrett98}.
	\end{rem}

	\begin{rem}
For the CH model with Flory Huggins energy potential, there have been some works to address the energy stability in the existing literature~\cite{jeong16, LiH2017, LiX16, peng17b, yang17c}. However, the positivity-preserving property has not been theoretically justified for these numerical works, so that the existence of the numerical solutions in these works is not available at a theoretical level.
	\end{rem}

	\section{Optimal rate convergence analysis in $\ell^\infty (0,T; H^{-1}) \cap \ell^2 (0,T; H^1)$}
	\label{sec:convergence}

For simplicity of presentation, we assume ${\cal M} \equiv 1$ in this section; the convergence analysis for the non-constant mobility case will be considered in future works.

Let $\Phi$ be the exact solution for the Cahn-Hilliard flow \eqref{CH equation-0} -- \eqref{CH-mu-0}. With initial data with sufficient regularity, we could assume that the exact solution has regularity of class $\mathcal{R}$:
	\begin{equation}
\Phi \in \mathcal{R} := H^2 \left(0,T; C_{\rm per}(\Omega)\right) \cap H^1 \left(0,T; C^2_{\rm per}(\Omega)\right) \cap L^\infty \left(0,T; C^6_{\rm per}(\Omega)\right).
	\label{assumption:regularity.1}
	\end{equation}
Define $\Phi_N (\, \cdot \, ,t) := {\cal P}_N \Phi (\, \cdot \, ,t)$, the (spatial) Fourier projection of the exact solution into ${\cal B}^K$, the space of trigonometric polynomials of degree to and including  $K$.  The following projection approximation is standard: if $\Phi\in L^\infty(0,T;H^\ell_{\rm per}(\Omega))$, 
	\begin{equation}
\nrm{\Phi_N - \Phi}_{L^\infty(0,T;H^k)}
   \le C h^{\ell-k} \nrm{\Phi }_{L^\infty(0,T;H^\ell)},  \quad \forall \ 0 \le k \le \ell .
	\label{projection-est-0}
	\end{equation}
By $\Phi_N^m$, $\Phi^m$ we denote $\Phi_N(\, \cdot \, , t_m)$ and $\Phi(\, \cdot \, , t_m)$, respectively, with $T_m = m\cdot \dt$. Since $\Phi_N \in {\cal B}^K$, the mass conservative property is available at the discrete level:
	\begin{equation}
\overline{\Phi_N^m} = \frac{1}{|\Omega|}\int_\Omega \, \Phi_N ( \cdot, t_m) \, d {\bf x} = \frac{1}{|\Omega|}\int_\Omega \, \Phi_N ( \cdot, t_{m-1}) \, d {\bf x} = \overline{\Phi_N^{m-1}} ,  \quad \forall \ m \in\mathbb{N}.
	\label{mass conserv-1}
	\end{equation}
On the other hand, the solution of (\ref{scheme-CH_LOG-1}), (\ref{scheme-mu-0}) is also mass conservative at the discrete level:
	\begin{equation}
\overline{\phi^m} = \overline{\phi^{m-1}} ,  \quad \forall \ m \in \mathbb{N} .
	\label{mass conserv-2}
	\end{equation}
As indicated before, we use the mass conservative projection for the initial data:  $\phi^0 = {\mathcal P}_h \Phi_N (\, \cdot \, , t=0)$, that is
	\begin{equation}
\phi^0_{i,j,k} := \Phi_N (p_i, p_j, p_k, t=0) ,
	\label{initial data-0}
	\end{equation}	
The error grid function is defined as
	\begin{equation}
\tilde{\phi}^m := \mathcal{P}_h \Phi_N^m - \phi^m ,  \quad \forall \ m \in \left\{ 0 ,1 , 2, 3, \cdots \right\} .
	\label{CH_LOG-error function-1}
	\end{equation}
Therefore, it follows that  $\overline{\tilde{\phi}^m} =0$, for any $m \in \left\{ 0 ,1 , 2, 3, \cdots \right\}$,  so that the discrete norm $\nrm{ \, \cdot \, }_{-1,h}$ is well defined for the error grid function.

	\begin{thm}
	\label{thm:convergence}
Given initial data $\Phi(\, \cdot \, ,t=0) \in C^6_{\rm per}(\Omega)$, suppose the exact solution for Cahn-Hilliard equation \eqref{CH equation-0}-\eqref{CH-mu-0} is of regularity class $\mathcal{R}$. Then, provided $\dt$ and $h$ are sufficiently small, for all positive integers $n$, such that $t_n \le T$, we have
	\begin{equation}
\| \tilde{\phi}^n \|_{-1,h} +  \left( \varepsilon^2 \dt   \sum_{m=1}^{n} \| \nabla_h \tilde{\phi}^m \|_2^2 \right)^{1/2}  \le C ( \dt + h^2 ),
	\label{CH_LOG-convergence-0}
	\end{equation}
where $C>0$ is independent of $n$, $\dt$, and $h$.
	\end{thm}

	\begin{proof}
A careful consistency analysis indicates the following truncation error estimate:
\begin{equation}
   \frac{\Phi_N^{n+1} - \Phi_N^n}{\dt}  = \Delta_h
 \left( \ln (1+\Phi_N^{n+1}) - \ln (1-\Phi_N^{n+1}) - \theta_0 \Phi_N^n - \varepsilon^2 \Delta_h \Phi_N^{n+1}  \right) + \tau^n ,
	\label{CH_LOG-consistency-1}
	\end{equation}
with $\| \tau^n \|_{-1,h} \le C (\dt + h^2)$. Observe that in equation~\eqref{CH_LOG-consistency-1}, and from this point forward, we drop the operator $\mathcal{P}_h$, which should appear in front of $\Phi_N$, for simplicity.

Subtracting the numerical scheme (\ref{scheme-CH_LOG-1}) from (\ref{CH_LOG-consistency-1}) gives
\begin{eqnarray}
   \frac{\tilde{\phi}^{n+1} - \tilde{\phi}^n}{\dt}  &=& \Delta_h
 \Bigl( ( \ln (1+\Phi_N^{n+1}) - \ln (1+\phi^{n+1})) - ( \ln (1-\Phi_N^{n+1}) - \ln (1-\phi^{n+1}))  \nonumber
\\
  &&  \quad
   - \theta_0 \tilde{\phi}^{n+1} - \varepsilon^2 \Delta_h \tilde{\phi}^{n+1}  \Bigr)
    + \tau^n .
	\label{CH_LOG-consistency-2}
	\end{eqnarray}
	
Since the numerical error function has zero-mean, we see that $(-\Delta_h)^{-1} \tilde{\phi}^m$ is well-defined, for any $k \ge 0$. Taking a discrete inner product with (\ref{CH_LOG-consistency-2}) by $2 (-\Delta_h)^{-1} \tilde{\phi}^{n+1}$ yields
	\begin{eqnarray}
\| \tilde{\phi}^{n+1} \|_{-1,h}^2 &-& \| \tilde{\phi}^n \|_{-1,h}^2 + \| \tilde{\phi}^{n+1} - \tilde{\phi}^n \|_{-1,h}^2  - 2 \varepsilon^2 \dt \ciptwo{ \tilde{\phi}^{n+1} }{ \Delta_h \tilde{\phi}^{n+1} }
	\nonumber
	\\
&+&  2 \dt \ciptwo{ \ln (1+\Phi_N^{n+1}) - \ln (1+\phi^{n+1}) }{ \tilde{\phi}^{n+1} }
	\nonumber
	\\
&-& 2 \dt \ciptwo{ \ln (1-\Phi_N^{n+1}) - \ln (1-\phi^{n+1}) }{ \tilde{\phi}^{n+1} }
	\nonumber
	\\
& & \hspace{0.25in}=  2 \theta_0 \dt \ciptwo{ \tilde{\phi}^n }{ \tilde{\phi}^{n+1} }  + 2 \dt \ciptwo{ \tau^n }{ \tilde{\phi}^{n+1} }.
	\label{CH_LOG-convergence-1}
	\end{eqnarray}
The estimate for the term associated with the surface diffusion is straightforward:
\begin{eqnarray}
  - \langle \tilde{\phi}^{n+1} , \Delta_h \tilde{\phi}^{n+1} \rangle  = \| \nabla_h \tilde{\phi}^{n+1} \|_2^2 .  \label{CH_LOG-convergence-2}
\end{eqnarray}
For the nonlinear inner product, the fact that $-1 < \phi^{n+1} < 1$, $-1 < \Phi^{n+1} < 1$ (at a point-wise level) yields the following result:
	\begin{eqnarray}
\ciptwo{ \ln (1+\Phi_N^{n+1}) - \ln (1+\phi^{n+1}) }{ \tilde{\phi}^{n+1} } &\ge& 0 ,
	\label{CH_LOG-convergence-3-1}
	\\
- \ciptwo{ \ln (1-\Phi_N^{n+1}) - \ln (1-\phi^{n+1}) }{ \tilde{\phi}^{n+1} }  &\ge& 0,
	\label{CH_LOG-convergence-3-2}
	\end{eqnarray}
due to the fact that $\ln$ is an increasing function. In other words, the convexity of the nonlinear term plays an essential role in this analysis. For the inner product associated with the concave part, the following estimate is derived:
	\begin{eqnarray}
 2 \theta_0 \ciptwo{ \tilde{\phi}^n }{ \tilde{\phi}^{n+1} } &\le& 2 \theta_0 \| \tilde{\phi}^n \|_{-1,h}  \| \nabla_h \tilde{\phi}^{n+1} \|_2
	\nonumber
	\\
& \le & \theta_0^2 \varepsilon^{-2} \| \tilde{\phi}^n \|_{-1,h}^2  + \varepsilon^2 \| \nabla_h \tilde{\phi}^{n+1} \|_2 .
	\label{CH_LOG-convergence-4}
	\end{eqnarray}
The term associated with the truncation error can be controlled in a standard way:
	\begin{equation}
2 \ciptwo{\tau^n }{ \tilde{\phi}^{n+1} }  \le  2 \| \tau^n \|_{-1,h}  \| \nabla_h \tilde{\phi}^{n+1} \|_2  \le  2 \varepsilon^{-2} \| \tau^n \|_{-1,h}^2 + \frac{\varepsilon^2}{2}  \| \nabla_h \tilde{\phi}^{n+1} \|_2^2 .
	\label{CH_LOG-convergence-5}
	\end{equation}
Using estimates  (\ref{CH_LOG-convergence-2}) -- (\ref{CH_LOG-convergence-5}) in (\ref{CH_LOG-convergence-1}) yields
	\begin{eqnarray}
\| \tilde{\phi}^{n+1} \|_{-1,h}^2 - \| \tilde{\phi}^n \|_{-1,h}^2 & +& \frac{\varepsilon^2}{2} \dt \| \nabla_h \tilde{\phi}^{n+1} \|_2^2
	\nonumber
	\\
&\le& \theta_0^2 \varepsilon^{-2} \dt \| \tilde{\phi}^n \|_{-1,h}^2 + 2 \varepsilon^{-2} \dt \| \tau^n \|_{-1,h}^2  .
	\label{CH_LOG-convergence-6}
	\end{eqnarray}
Finally, an application of a discrete Gronwall inequality results in the desired convergence estimate:
	\begin{equation}
\| \tilde{\phi}^{n+1} \|_{-1,h} + \left( \varepsilon^2 \dt \sum_{k=0}^{n+1} \| \nabla_h \tilde{\phi}^m \|_2^2 \right)^{1/2}  \le C ( \dt + h^2) ,
	\label{CH_LOG-convergence-7}
	\end{equation}
where  $C>0$ is independent of $\dt$, $h$, and $n$. This completes the proof of the theorem.
	\end{proof}
	
\begin{rem}
For the Cahn-Hilliard equation with logarithmic potential, there have been some existing works of error estimate~\cite{barrett95, barrett96, barrett01} in the framework of finite element analysis, with implicit Euler method in the temporal discretization. Again, the time step constraint $\dt \le  \frac{4 \varepsilon^2}{\theta_0^2}$ has to be imposed to ensure the positivity-preserving property of the numerical scheme, while no constraint is needed in the convergence analysis of our proposed scheme.
\end{rem}

\section{The second order numerical scheme} \label{sec:BDF2}

We propose the following second order scheme for the CH equation~\eqref{CH equation-0}-\eqref{CH-mu-0}: given $\phi^n, \phi^{n-1} \in \mathcal{C}_{\rm per}$, find $\phi^{n+1},\mu^{n+1}\in  \mathcal{C}_{\rm per}$, such that
	\begin{equation}
\frac{\frac32 \phi^{n+1} - 2 \phi^n + \frac12 \phi^{n-1}}{\dt}  =  \nabla_h \cdot ( \widehat{\cal M}^{n+1} \nabla_h \mu^{n+1} )  ,
	\label{BDF2-CH_LOG-1}
	\end{equation}
where
	\begin{align} \label{scheme-mu-BDF2-1} 	
\mu^{n+1} = & \ln (1+\phi^{n+1}) - \ln (1-\phi^{n+1}) - \theta_0 \check{\phi}^{n+1} - A \dt \Delta_h ( \phi^{n+1} - \phi^n ) - \varepsilon^2 \Delta_h \phi^{n+1} ,
\\
  \check{\phi}^{n+1} = & 2 \phi^n - \phi^{n-1} ,   \nonumber
	\end{align}
and the discrete mobility function is defined at the face center points in a similar way as in~\eqref{scheme-CH_LOG-mobility-1}: $\widehat{\cal M}_{i+\hf,j,k}^{n+1} = {\cal M}(A_x \check{\phi}^{n+1}_{i+\hf,j,k})$, $\widehat{\cal M}_{i,j+\hf,k}^{n+1} = {\cal M}(A_y \check{\phi}^{n+1}_{i,j+\hf,k})$, $\widehat{\cal M}_{i,j,k+\hf}^{n+1} = {\cal M}(A_z \check{\phi}^{n+1}_{i,j,k+\hf})$.


In the case of constant mobility ${\cal M} (\phi) \equiv 1$, the positivity-preserving property is established in the following theorem.

	\begin{thm}
	\label{CH-BDF2-positivity}
Assume that ${\cal M} (\phi) \equiv 1$. Given $\phi^k \in\mathcal{C}_{\rm per}$, with $\nrm{\phi^k}_\infty \le M$, $k=n, n-1$, for some $M >0$,  and $\left|\overline{\phi^k}\right| = \left|\overline{\phi^{n-1}} \right| < 1$, there exists a unique solution $\phi^{n+1}\in\mathcal{C}_{\rm per}$ to \eqref{BDF2-CH_LOG-1}, with $\phi^{n+1}-\overline{\phi^n}\in\mathring{\mathcal{C}}_{\rm per}$ and  $\nrm{\phi^{n+1}}_\infty < 1$.
	\end{thm}

	\begin{proof}
We follow the notations in the proof of Theorem~\ref{CH-positivity}. The numerical solution of \eqref{BDF2-CH_LOG-1} is a minimizer of the following discrete energy functional over the admissible set $A_h$:
	\begin{eqnarray}
\mathcal{J}^{n, (2)} (\phi) &:=& \frac{1}{3 \dt} \nrm{ \frac32 \phi - 2 \phi^n + \frac12 \phi^{n-1} }_{-1,h}^2  \nonumber
\\
  &&
 + \ciptwo{ 1+ \phi }{ \ln (1+\phi) } + \ciptwo{ 1-\phi }{  \ln (1-\phi) }  \nonumber
\\
  &&
 + \frac{\varepsilon^2 + A \dt}{2} \| \nabla_h \phi \|_2^2  + \ciptwo{ \phi }{ A\dt \Delta_h \phi^n - \theta_0\check{\phi}^{n+1} } .
	\label{CH_LOG-BDF2-positive-1}
	\end{eqnarray}
Of course, $\mathcal{J}^{n, (2)}$ is strictly convex over $A_h$. Again, such a minimization problem is equivalent to the following transformed functional over $\mathring{A}_h$:
	\begin{eqnarray}
\mathcal{F}^{n, (2)} (\varphi) &:=& \mathcal{J}^{n, (2)} (\varphi + \overline{\phi}_0)
\nonumber
	\\
&=&  \frac{1}{3 \dt} \nrm{ \frac32 (\varphi + \overline{\phi}_0) - 2 \phi^n + \frac12 \phi^{n-1} }_{-1,h}^2  \nonumber
\\
  &&
  + \ciptwo{ 1+ \varphi + \overline{\phi}_0 }{ \ln ( 1 + \varphi + \overline{\phi}_0 )}
    + \ciptwo{1- \varphi - \overline{\phi}_0 }{  \ln (1-\varphi - \overline{\phi}_0)  } \nonumber
\\
  && + \frac{\varepsilon^2 + A \dt}{2} \nrm{ \nabla_h \varphi }_2^2
  + \ciptwo{ \varphi + \overline{\phi}_0 }{ A\dt \Delta_h \phi^n - \theta_0\check{\phi}^{n+1} }  .
	\label{CH_LOG-BDF2-positive-1-b}
	\end{eqnarray}
To obtain the existence of a minimizer for $\mathcal{F}^{n, (2)}$ over $\mathring{A}_h$, we consider the closed domain $\mathring{A}_{h,\delta}$ for $0 < \delta < \frac12$, as defined by~\eqref{CH_LOG-positive-2}.
There exists a (not necessarily  unique) minimizer of $\mathcal{F}^{n, (2)}$ over $\mathring{A}_{h, \delta}$, and we have to prove such a minimizer could not occur on the boundary of $\mathring{A}_{h,\delta}$, if $\delta$ is sufficiently small. 
To get a contradiction, suppose that the minimizer of $\mathcal{F}^{n, (2)}$, call it $\varphi^\star$ occurs at a boundary point of $\mathring{A}_{h,\delta}$.  There is at least one grid point $\vec{\alpha}_0 = (i_0,j_0,k_0)$ such that $|\varphi^\star_{\vec{\alpha}_0}+\overline{\phi}_0| = 1 -\delta$. Similarly, we assume that $\varphi^\star_{\vec{\alpha}_0}+\overline{\phi}_0 = \delta - 1$, so that the grid function $\varphi^\star$ has a global minimum at $\vec{\alpha}_0$, and $\vec{\alpha}_1 = (i_1,j_1,k_1)$ is a grid point at which $\varphi^\star$ achieves its maximum. 
Meanwhile, for all $\psi\in \mathring{\mathcal{C}}_{\rm per}$, the directional derivative becomes
	\begin{align*}
 d_s \mathcal{F}^{n, (2)} (\varphi^\star+s\psi)|_{s=0} = & \  \ciptwo{ \ln (1+\varphi^\star+\overline{\phi}_0) - \ln (1-\varphi^\star-\overline{\phi}_0)}{\psi}
	\\
& +  \ciptwo{ A \dt \Delta_h \phi^n - \theta_0 \check{\phi}^{n+1} }{\psi}
-  (\varepsilon^2 + A \dt )  \ciptwo{ \Delta_h \varphi^\star}{\psi}
	\\
& + \frac{1}{\Delta t}\ciptwo{(-\Delta_h)^{-1}\left( \frac32 ( \varphi^\star +\overline{\phi}_0 ) - 2 \phi^n + \frac12 \phi^{n-1} \right)}{\psi}.
	\end{align*}
In more details, this derivative may be expressed as
	\begin{eqnarray}
\frac{1}{h^3} d_s \mathcal{F}^{n, (2)} (\varphi^\star+s\psi)|_{s=0}  &=&  \ln (1+ \varphi^\star_{\vec{\alpha}_0}+\overline{\phi}_0) - \ln (1-\varphi^\star_{\vec{\alpha}_0}-\overline{\phi}_0)
	\nonumber
	\\
&& - \ln (1+ \varphi^\star_{\vec{\alpha}_1}+\overline{\phi}_0) + \ln (1-\varphi^\star_{\vec{\alpha}_1}-\overline{\phi}_0)
	\nonumber
	\\
&& - \theta_0 ( \check{\phi}^{n+1}_{\vec{\alpha}_0} - \check{\phi}^{n+1}_{\vec{\alpha}_1} ) + A \dt ( \Delta_h \phi^n_{\vec{\alpha}_0} -  \Delta_h \phi^n_{\vec{\alpha}_1} ) \nonumber
\\
&&  - ( \varepsilon^2 + A \dt ) ( \Delta_h \varphi^\star_{\vec{\alpha}_0} -  \Delta_h \varphi^\star_{\vec{\alpha}_1} )
	\nonumber
	\\
&& + \frac{1}{\dt}  (-\Delta_h)^{-1} ( \frac32 ( \varphi^\star +\overline{\phi}_0 ) - 2 \phi^n + \frac12 \phi^{n-1} )_{\vec{\alpha}_0}
	\nonumber
	\\
&& - \frac{1}{\dt}(-\Delta_h)^{-1} ( \frac32 ( \varphi^\star +\overline{\phi}_0 ) - 2 \phi^n + \frac12 \phi^{n-1} )_{\vec{\alpha}_1} .
	\label{CH_LOG-BDF2-positive-4}
	\end{eqnarray}
Furthermore, 
the following estimates are derived
	\begin{eqnarray}
\Delta_h \phi^\star_{\vec{\alpha}_0} &\ge& 0 ,  \quad \Delta_h \phi^\star_{\vec{\alpha}_1} \le 0  , \label{CH_LOG-BDF2-positive-6-1}
\\
-6 M &\le& \check{\phi}^{n+1}_{\vec{\alpha}_0} - \check{\phi}^{n+1}_{\vec{\alpha}_1} \le 6M ,  \label{CH_LOG-BDF2-positive-6-2}
\\
  \Delta_h \phi^n_{\vec{\alpha}_0} &\le& \frac{12 M}{h^2} ,  \quad
 \Delta_h \phi^n_{\vec{\alpha}_1} \ge - \frac{12 M}{h^2} ,
 \label{CH_LOG-BDF2-positive-6-3}
\\
- 5 C_1 &\le& (-\Delta_h)^{-1} ( \frac32 ( \varphi^\star +\overline{\phi}_0 ) - 2 \phi^n + \frac12 \phi^{n-1} )_{\vec{\alpha}_0}   \nonumber
\\
  &&
  - \frac{1}{\dt}(-\Delta_h)^{-1} ( \frac32 ( \varphi^\star +\overline{\phi}_0 ) - 2 \phi^n + \frac12 \phi^{n-1} )_{\vec{\alpha}_1} \le  5 C_1 ,
	\label{CH_LOG-BDF2-positive-6-4}
	\end{eqnarray}
in which we have repeatedly made use of the fact that $\| \phi^k \|_\infty \le M$, $k=n, n-1$, as well as the application of Lemma~\ref{CH-positivity-Lem-0}. Subsequently,  a substitution of (\ref{CH_LOG-BDF2-positive-6-1}) -- (\ref{CH_LOG-BDF2-positive-6-4}) and \eqref{CH_LOG-positive-5} into (\ref{CH_LOG-BDF2-positive-4}) yields the following bound:
	\begin{equation}
\frac{1}{h^3} d_s \mathcal{F}^{n, (2)} (\varphi^\star+s\psi)|_{s=0} \le \ln \frac{\delta}{2 - \delta} - \ln \frac{1+\overline{\phi}_0}{1-\overline{\phi}_0}  + 6 M \theta_0 + 12 M \dt h^{-2} + 10 C_1 \dt^{-1} .
	\label{CH_LOG-BDF2-positive-10}
	\end{equation}
The rest analysis follows the same arguments as in the proof of Theorem~\ref{CH-positivity}; the details are left to interested readers.
\end{proof}

	\begin{rem}
Again, for the second order scheme, a careful calculation implies that  $C_8 = O (\dt^{-1} + \dt h^{-2})$, which becomes singular as $\dt, h\to 0$. Even so, since the values of $h$ and $\dt$ are fixed, a $\delta\in(0,\hf)$ exists so that the size of $C_8$ is not an issue.
	\end{rem}

The non-constant mobility case could be analyzed in the same fashion; we state the result below, and the technical details are left to interested readers.

	\begin{thm}
	\label{CH-mobility-BDF2-positivity}
Assume that ${\cal M} (x) \ge  {\cal M}_0 >0 $, for all $x\in[-1,1]$. Given $\phi^k \in\mathcal{C}_{\rm per}$, with $\nrm{\phi^k} \le M$, $k=n, n-1$, for some $M >0$,  and $\left|\overline{\phi^n}\right| = \left|\overline{\phi^{n-1}}\right| < 1$, there exists a unique solution $\phi^{n+1}\in\mathcal{C}_{\rm per}$ to \eqref{BDF2-CH_LOG-1}, with $\phi^{n+1}-\overline{\phi^n}\in\mathring{\mathcal{C}}_{\rm per}$ and  $\nrm{\phi^{n+1}}_\infty < 1$.
	\end{thm}
	
In the case of constant mobility ${\cal M} (\phi) \equiv 1$, a modified energy stability is available for the second order BDF scheme~\eqref{BDF2-CH_LOG-1}, provided that $A \ge \frac{1}{16}$. 	

	\begin{thm}
	\label{CH-BDF2-energy stability}
Suppose ${\cal M} (\phi) \equiv 1$. With the same assumptions as in Theorem~\ref{CH-mobility-energy stability}, we have the stability analysis of the following modified energy functional for the proposed numerical scheme \eqref{BDF2-CH_LOG-1}:
\begin{eqnarray}
  &&
   \tilde{E}_h (\phi^{n+1}, \phi^n)  \le  \tilde{E}_h (\phi^n, \phi^{n-1}) ,  \quad
   \mbox{with}  \label{CH-BDF2-stability-0}
\\
  &&
    \tilde{E}_h (\phi^{n+1}, \phi^n) = E_h (\phi^{n+1})
  + \frac{1}{4 \dt} \| \phi^{n+1} - \phi^n \|_{-1,h}^2
  + \frac12 \| \phi^{n+1} - \phi^n \|_2^2  ,   \label{mod energy-BDF2-1}
\end{eqnarray}
for any $\dt, h >0$, provided that $A \ge \frac{1}{16}$.
	\end{thm}
	
\begin{proof}
By taking an inner product with~\eqref{BDF2-CH_LOG-1} by $(-\Delta_h)^{-1} (\phi^{n+1} - \phi^n)$, we could derive the following inequalities:
\begin{eqnarray}
  &&
  \left\langle  \frac{\frac32 \phi^{n+1} - 2 \phi^n + \frac12 \phi^{n-1}}{\dt} ,
  (-\Delta_h)^{-1} (\phi^{n+1} - \phi^n)  \right\rangle_\Omega
  \nonumber
\\
  &&  \qquad
  = \frac{3}{2 \dt} \| \phi^{n+1} - \phi^n \| _{-1,h}^2
  - \frac12 \langle \phi^{n+1} - \phi^n , \phi^n - \phi^{n-1}  \rangle_{-1, h} \nonumber
\\
  && \qquad
  \ge \frac{1}{\dt}  \left(
    \frac54 \| \phi^{n+1} - \phi^n \|_{-1, h}^2
  - \frac14 \| \phi^n - \phi^{n-1} \|_{-1, h}^2  \right) ,
    \label{CH-BDF2-stability-1}
\\
  &&
   \left\langle  - \Delta_h ( \ln (1+\phi^{n+1}) )  ,
  (-\Delta_h)^{-1} (\phi^{n+1} - \phi^n)  \right\rangle_\Omega
  = \left\langle  \ln (1+\phi^{n+1})   ,  \phi^{n+1} - \phi^n  \right\rangle_\Omega  \nonumber
\\
  &&  \qquad
  \ge \ciptwo{ 1+ \phi^{n+1} }{ \ln (1+\phi^{n+1}) }
  - \ciptwo{ 1+ \phi^n }{ \ln (1+\phi^n) }   ,
    \label{CH-BDF-stability-2-1}
\\
  &&
   \left\langle    \Delta_h ( \ln (1-\phi^{n+1}) )  ,
  (-\Delta_h)^{-1} (\phi^{n+1} - \phi^n)  \right\rangle_\Omega
  = - \left\langle  \ln (1-\phi^{n+1})   ,  \phi^{n+1} - \phi^n  \right\rangle_\Omega  \nonumber
\\
  &&  \qquad
  \ge - \ciptwo{ 1- \phi^{n+1} }{ \ln (1-\phi^{n+1}) }
  + \ciptwo{ 1- \phi^n }{ \ln (1-\phi^n) }   ,
    \label{CH-BDF-stability-2-2}
\\
  &&
  \left\langle  \Delta_h^2 \phi^{n+1}  ,
  (-\Delta_h)^{-1} (\phi^{n+1} - \phi^n)  \right\rangle_\Omega
  =  \left\langle  \nabla_h \phi^{n+1}  ,
   \nabla_h (\phi^{n+1} - \phi^n)  \right\rangle_\Omega  \nonumber
\\
  &&  \qquad
   =  \frac12 \left(  \| \nabla_h \phi^{n+1} \|_2^2
  - \| \nabla_h \phi^n \|_2^2
  + \| \nabla_h ( \phi^{n+1} - \phi^n ) \|_2^2 \right)  ,
    \label{CH-BDF2-stability-3}
\\
  &&
  \dt \left\langle  \Delta_h^2 ( \phi^{n+1} - \phi^n ) ,
  (-\Delta_h)^{-1} (\phi^{n+1} - \phi^n)  \right\rangle_\Omega
  =  \dt \| \nabla_h ( \phi^{n+1} - \phi^n ) \|_2^2  ,
    \label{CH-BDF2-stability-4}
\\
  &&
  \left\langle  \Delta_h  ( 2 \phi^n - \phi^{n-1})  ,
  (-\Delta_h)^{-1} (\phi^{n+1} - \phi^n)  \right\rangle_\Omega
  = - \left\langle  2 \phi^n - \phi^{n-1}  ,
  \phi^{n+1} - \phi^n)  \right\rangle_\Omega   \nonumber
\\
  &&  \qquad
   \ge - \frac12 \left(  \| \phi^{n+1} \|_2^2 - \| \phi^n \|_2^2   \right)
  - \frac12 \| \phi^n - \phi^{n-1} \|_2^2 ,
    \label{CH-BDF2-stability-5}
\end{eqnarray}
in which~\eqref{CH-BDF-stability-2-1}, \eqref{CH-BDF-stability-2-2} are based on the convexity of $(1 + \phi) \ln (1+ \phi)$, $(1 - \phi) \ln (1- \phi)$, respectively. Meanwhile, an application of Cauchy inequality indicates the following estimate:
\begin{equation}
   \frac{1}{\dt} \| \phi^{n+1} - \phi^n \|_{-1,h}^2
   + A \dt \| \nabla_h ( \phi^{n+1} - \phi^n ) \|_2^2
   \ge 2 A^{1/2}  \| \phi^{n+1} - \phi^n \|_2^2 .
  \label{CH-BDF2-stability-6}
\end{equation}
Therefore, a combination of~\eqref{CH-BDF2-stability-1}-\eqref{CH-BDF2-stability-5} and~\eqref{CH-BDF2-stability-6} yields
\begin{eqnarray}
  &&
  E_h (\phi^{n+1}) - E_h (\phi^n)
  + \frac{1}{4 \dt}  \left(  \| \phi^{n+1} - \phi^n \|_{-1, h}^2
  - \| \phi^n - \phi^{n-1} \|_{-1, h}^2  \right)  \nonumber
\\
  &&
  + \frac12 \left( \| \phi^{n+1} - \phi^n \|_2^2 - \| \phi^n - \phi^{n-1} \|_2^2  \right)
  \le ( - 2 A^{1/2} + \frac12 ) \| \phi^{n+1} - \phi^n \|_2^2  \le 0 ,
  \label{scheme-BDF-stability-7}
\end{eqnarray}
provided that $A \ge \frac{1}{16}$. Therefore, by denoting a modified energy as given by~\eqref{mod energy-BDF2-1}, we get the energy estimate~\eqref{CH-BDF2-stability-0}. This completes the proof of Theorem~\ref{CH-BDF2-energy stability}.
\end{proof}

With the same assumption that ${\cal M} (\phi) \equiv 1$, the convergence result is stated in the following theorem.

	\begin{thm}
	\label{thm:BDF2-convergence}
Given initial data $\Phi(\, \cdot \, ,t=0) \in C^6_{\rm per}(\Omega)$, suppose the exact solution for Cahn-Hilliard equation \eqref{CH equation-0}-\eqref{CH-mu-0} is of regularity class $\mathcal{R}_2 := H^3 \left(0,T; C_{\rm per}(\Omega)\right) \cap H^3 \left(0,T; C^2_{\rm per}(\Omega)\right) \cap L^\infty \left(0,T; C^6_{\rm per}(\Omega)\right)$. Then, provided $\dt$ and $h$ are sufficiently small, for all positive integers $n$, such that $t_n \le T$, we have the following convergence estimate for the numerical solution~\eqref{BDF2-CH_LOG-1}
	\begin{equation}
\| \tilde{\phi}^n \|_{-1,h} +  \left( \varepsilon^2 \dt   \sum_{m=1}^{n} \| \nabla_h \tilde{\phi}^m \|_2^2 \right)^{1/2}  \le C ( \dt^2 + h^2 ),
	\label{CH_LOG-BDF2-convergence-0}
	\end{equation}
where $C>0$ is independent of $n$, $\dt$, and $h$.
	\end{thm}

	\begin{proof}
A careful consistency analysis indicates the following truncation error estimate:
\begin{eqnarray}
   \frac{\frac32 \Phi_N^{n+1} - 2 \Phi_N^n + \frac12 \Phi_N^{n-1}}{\dt}  &=& \Delta_h
 \Bigl( \ln (1+\Phi_N^{n+1}) - \ln (1-\Phi_N^{n+1}) - \theta_0 \check{\Phi}_N^{n+1} - \varepsilon^2 \Delta_h \Phi_N^{n+1}  \nonumber
\\
  &&
  - A \dt \Delta_h ( \Phi_N^{n+1} - \Phi_N^n ) \Bigr) + \tau^n ,
	\label{CH_LOG-BDF2-consistency-1}
	\end{eqnarray}
with $\check{\Phi}_N^n = 2 \Phi_N^n - \Phi_N^{n-1}$, $\| \tau^n \|_{-1,h} \le C (\dt^2 + h^2)$. 
In turn, subtracting the numerical scheme~\eqref{BDF2-CH_LOG-1} from~\eqref{CH_LOG-consistency-1} gives
\begin{eqnarray}
   \frac{\frac32 \tilde{\phi}^{n+1} - 2 \tilde{\phi}^n + \frac12 \tilde{\phi}^{n-1}}{\dt}  &=&
   \Delta_h
 \Bigl( ( \ln (1+\Phi_N^{n+1}) - \ln (1+\phi^{n+1}))  \nonumber
\\
  && \quad
  - ( \ln (1-\Phi_N^{n+1}) - \ln (1-\phi^{n+1}))
  - \theta_0 \tilde{\check{\phi}}^{n+1} \nonumber
\\
  &&  \quad
   - \varepsilon^2 \Delta_h \tilde{\phi}^{n+1}
   - A \dt \Delta_h ( \tilde{\phi}^{n+1} - \tilde{\phi}^n ) \Bigr)
    + \tau^n ,
	\label{CH_LOG-BDF2-consistency-2}
	\end{eqnarray}
with $\tilde{\check{\phi}}^{n+1} = 2 \tilde{\phi}^n - \tilde{\phi}^{n-1}$.
Taking a discrete inner product with~\eqref{CH_LOG-BDF2-consistency-2} by $2 (-\Delta_h)^{-1} \tilde{\phi}^{n+1}$ yields
	\begin{eqnarray}
	&&
\left\langle 3 \tilde{\phi}^{n+1} - 4 \tilde{\phi}^n + \tilde{\phi}^{n-1} ,
  \tilde{\phi}^{n+1}  \right\rangle_{-1, h}
  - 2 \varepsilon^2 \dt \ciptwo{ \tilde{\phi}^{n+1} }{ \Delta_h \tilde{\phi}^{n+1} }
	\nonumber
	\\
&+&  2 \dt \ciptwo{ \ln (1+\Phi_N^{n+1}) - \ln (1+\phi^{n+1}) }{ \tilde{\phi}^{n+1} }
\nonumber
\\
  &-&
   2 \dt \ciptwo{ \ln (1-\Phi_N^{n+1}) - \ln (1-\phi^{n+1}) }{ \tilde{\phi}^{n+1} }
	\nonumber
	\\
&-&
 2 A \dt \langle \Delta_h ( \tilde{\phi}^{n+1} - \tilde{\phi}^n ) , \tilde{\phi}^{n+1} \rangle_\Omega
 =  2 \theta_0 \dt \ciptwo{ \tilde{\phi}^n }{ \tilde{\phi}^{n+1} }  + 2 \dt \ciptwo{ \tau^n }{ \tilde{\phi}^{n+1} }.
	\label{CH_LOG-BDF2-convergence-1}
	\end{eqnarray}
For the temporal derivative stencil, the following identity is valid:
      \begin{eqnarray}
      \left\langle 3 \tilde{\phi}^{n+1} - 4 \tilde{\phi}^n + \tilde{\phi}^{n-1} ,
  \tilde{\phi}^{n+1}  \right\rangle_{-1, h}
	&=& \frac12 \Bigl( \| \tilde{\phi}^{n+1} \|_{-1, h}^2 - \| \tilde{\phi}^n \|_{-1, h}^2
	\nonumber
\\
  &&
	+ \| 2 \tilde{\phi}^{n+1} - \tilde{\phi}^n \|_{-1, h}^2
	- \| 2 \tilde{\phi}^n - \tilde{\phi}^{n-1} \|_{-1, h}^2  \nonumber
\\
  &&
  + \| \tilde{\phi}^{n+1}  - 2 \tilde{\phi}^n + \tilde{\phi}^{n-1} \|_{-1, h}^2 \Bigr) .
  \label{CH_LOG-BDF2-convergence-2}
	\end{eqnarray}
The estimates for the terms associated with the surface diffusion, the nonlinear product and the truncation error follow exactly the same way as in~\eqref{CH_LOG-convergence-2}, \eqref{CH_LOG-convergence-3-1}, \eqref{CH_LOG-convergence-3-2}, \eqref{CH_LOG-convergence-5}, respectively. For the concave expansive error term, a similar inequality is available:
	\begin{eqnarray}
 2 \theta_0 \ciptwo{ \tilde{\check{\phi}}^{n+1} }{ \tilde{\phi}^{n+1} } &\le& 2 \theta_0 \| \tilde{\check{\phi}}^{n+1} \|_{-1,h}  \| \nabla_h \tilde{\phi}^{n+1} \|_2
  \le \theta_0^2 \varepsilon^{-2} \| \tilde{\check{\phi}}^{n+1} \|_{-1,h}^2  + \varepsilon^2 \| \nabla_h \tilde{\phi}^{n+1} \|_2
	\nonumber
	\\
& \le & \theta_0^2 \varepsilon^{-2} ( 8 \| \tilde{\phi}^n \|_{-1,h}^2
 + 2  \| \tilde{\phi}^{n-1} \|_{-1,h}^2) + \varepsilon^2 \| \nabla_h \tilde{\phi}^{n+1} \|_2 .
	\label{CH_LOG-BDF2-convergence-4}
	\end{eqnarray}
In addition, the following identity could be derived for the artificial diffusion term:
\begin{eqnarray}
  &&
  - 2 \langle \Delta_h ( \tilde{\phi}^{n+1} - \tilde{\phi}^n ) , \tilde{\phi}^{n+1} \rangle_\Omega
  = 2 \langle \nabla_h ( \tilde{\phi}^{n+1} - \tilde{\phi}^n ) , \nabla_h \tilde{\phi}^{n+1} \rangle_\Omega  \nonumber
\\
  &=&
  \| \nabla_h \tilde{\phi}^{n+1} \|_2^2 - \| \nabla_h \tilde{\phi}^n \|_2^2
  +  \| \nabla_h ( \tilde{\phi}^{n+1} - \| \nabla_h \tilde{\phi}^n ) \|_2^2 .
    \label{CH_LOG-BDF2-convergence-5}
	\end{eqnarray}  	
Subsequently, a substitution of (\ref{CH_LOG-BDF2-convergence-2}) -- (\ref{CH_LOG-BDF2-convergence-5}), \eqref{CH_LOG-convergence-2}, \eqref{CH_LOG-convergence-3-1}, \eqref{CH_LOG-convergence-3-2} and \eqref{CH_LOG-convergence-5} into (\ref{CH_LOG-BDF2-convergence-1}) yields
	\begin{eqnarray}
	&&
\| \tilde{\phi}^{n+1} \|_{-1,h}^2 - \| \tilde{\phi}^n \|_{-1,h}^2
+ \| 2 \tilde{\phi}^{n+1} - \tilde{\phi}^n \|_{-1, h}^2
	- \| 2 \tilde{\phi}^n - \tilde{\phi}^{n-1} \|_{-1, h}^2  \nonumber
\\
&&
+ A \dt ( \| \nabla_h \tilde{\phi}^{n+1} \|_2^2 - \| \nabla_h \tilde{\phi}^n \|_2^2 )
+ \frac{\varepsilon^2}{2} \dt \| \nabla_h \tilde{\phi}^{n+1} \|_2^2
	\nonumber
	\\
&\le& 4 \theta_0^2 \varepsilon^{-2} ( 4 \| \tilde{\phi}^n \|_{-1,h}^2
 +  \| \tilde{\phi}^{n-1} \|_{-1,h}^2) + 4 \varepsilon^{-2} \dt \| \tau^n \|_{-1,h}^2  .
	\label{CH_LOG-BDF2-convergence-6}
	\end{eqnarray}
Finally, an application of a discrete Gronwall inequality results in the desired convergence estimate:
	\begin{equation}
\| \tilde{\phi}^{n+1} \|_{-1,h} + \Bigl( \varepsilon^2 \dt \sum_{k=0}^{n+1} \| \nabla_h \tilde{\phi}^m \|_2^2 \Bigr)^{1/2}  \le C ( \dt^2 + h^2) ,
	\label{CH_LOG-BDF2-convergence-7}
	\end{equation}
where  $C>0$ is independent of $\dt$, $h$, and $n$. This completes the proof of the Theorem~\ref{thm:BDF2-convergence}.
	\end{proof}

	\section{Numerical results}
	\label{sec:numerical results}

In this section we describe a simple multigrid solver for the proposed schemes, and we provided some tests that show the efficiency of the solver and the accuracy of the scheme. We demonstrate, in particular, the positivity of the solutions to the proposed Cahn-Hilliard scheme.

For the discussion of the numerical computations, we use a slightly different formulation of the Cahn-Hilliard equation, one that allows for a comparison with the so-called obstacle potential. Specifically, we will use the standard Ginzburg-Landau free energy $E[\phi] = \int_\Omega \left\{f(\phi) +\frac{\varepsilon^2}{2}|\nabla\phi|^2 \right\} d\mathbf{x}$, where $f(\phi) = f_c(\phi) - f_e(\phi)$ and
	\[
f_c(\phi) = \frac{1}{2\theta_0}\left[(1-\phi)\ln(1-\phi) +(1+\phi)\ln(1+\phi)\right] , \quad f_e(\phi) = \frac{1}{2}(\phi-1)(\phi+1) .
	\]
Importantly, as $\theta_0 \to \infty$, $f$ tends to the obstacle potential
	\[
f_{\rm obs}(\theta) = \left\{
	\begin{array}{ccc}
\frac{1}{2}(\phi-1)(\phi+1) & \mbox{if} & -1 < \phi < 1
	\\
\infty & \mbox{if} & |\phi|\ge 1
	\end{array}
\right.	 ,
	\]
which has been investigated elsewhere~\cite{blowey91,blowey92}. While we are only interested in the case of finite values of $\theta_0$, it is interesting to explore the effects of increasing $\theta_0$. For finite $\theta_0$, clearly $f_e'(\phi) = \phi$ and
	\[
f'_c(\phi) = \frac{1}{2\theta_0} \left[ \ln(1+\phi) - \ln(1-\phi)\right].
	\]
The Cahn-Hillard equation still takes the form \eqref{CH equation-0}, but with the chemical potential expressed as
	\[
 \mu = f_c'(\phi)-f_e'(\phi) - \varepsilon^2\Delta\phi.
	\]
As before, we assume that the mobility satisfies  ${\cal M} (x) \ge \mathcal{M}_0 >0$, for all $x\in[-1,1]$, for some $\mathcal{M}_0$, though as we have remarked, this can be relaxed.
	
	\subsection{Multgrid solver}
	\label{subsec:solver}
	
In this subsection, we describe a nonlinear full approximation storage (FAS) multigrid solver for the convex-concave decomposition scheme for the Cahn-Hilliard equation with logarithmic potential. The solver for the Allen-Cahn equation is simpler, and we omit its description. Our solver is similar in style to the one presented in~\cite{jeong16}, and it can be extended to the case of multi-component systems as in the article. For an alternative approach to the one taken here and in~\cite{jeong16}, see, for example, \cite{graser15}.

For our solver implementation, we  will need to regularize $f'_c$. This is due to the fact that our multigrid solver is not designed to guarantee the boundedness of the solution for arbitrary multigrid iterations, as we discuss below. Our solver will, however, converge to the correct bounded solution, provided the regularization is sufficiently small. We show this in our tests.

To effect the desired regularization, we modify the logarithm as follows: for a given $\delta\in (0,\nicefrac{1}{4})$ we define
	\[
\ln_\delta (\phi) = \left\{
	\begin{array}{ccc}
\ln(\phi) & \mbox{if} & \delta < \phi
	\\
\ln(\delta) + \frac{\phi-\delta}{\delta} & \mbox{if} & \phi \le \delta
	\end{array}
\right.	 .
	\]
The regularized logarithm, $\ln_\delta$ is defined for all values of $\phi$. Using this function, we define
	\[
f'_{c,\delta}(\phi) = \frac{1}{2\theta_0} \left[ \ln_\delta(1+\phi) - \ln_\delta(1-\phi)\right].
	\]
We then observe that $f'_c(\phi) = f'_{c,\delta}(\phi)$, for all $-1+\delta \le \phi \le 1-\delta$. Consequently, we can always take the value of $\delta$ to be small enough such that the theoretical solution to our scheme lies in this range of equivalence.

The first-order convex-concave decomposition (CS1) scheme \eqref{scheme-CH_LOG-1} in 2-D is equivalent to the following: find $\phi, \mu\in\mathcal{C}_{\rm per}$ whose components satisfy
	\begin{eqnarray}
\phi_{i,j} - \dt\, d_x\left({\cal M} \left(A_x\phi^m\right) D_x\mu \right)_{i,j} - \dt\, d_y\left( {\cal M} \left(A_y\phi^m\right) D_y\mu\right)_{i,j} &=& \phi_{i,j}^m  ,
	\label{disc-ch-1}
	\\
\mu_{i,j} - f_{c,\delta}'\left(\phi_{i,j}\right) +\epsilon^2\Delta_h\phi_{i,j} &=& -\phi_{i,j}^m  ,
	\label{disc-ch-2}
	\end{eqnarray}
where we have dropped the time superscripts $m+1$ on the unknowns.  The 3-D equations are similar, and they are omitted for simplicity. For the sake of comparison, the standard backward Euler scheme (BE) is
	\begin{eqnarray}
\phi_{i,j} - \dt\, d_x\left({\cal M} \left(A_x\phi \right) D_x\mu \right)_{i,j} - \dt\, d_y\left( {\cal M} \left(A_y\phi \right) D_y\mu\right)_{i,j} &=& \phi_{i,j}^m  ,
	\label{disc-ch-be-1}
	\\
\mu_{i,j} - f_{c,\delta}'\left(\phi_{i,j}\right) +\phi_{i,j}  +\epsilon^2\Delta_h\phi_{i,j} &=& 0 .
	\label{disc-ch-be-2}
	\end{eqnarray}
We note that, for solvability and stability considerations, the sign of the linear term ($\phi$) in the chemical potential equation \eqref{disc-ch-be-2} is problematical. However, this scheme is solvable with a mild time step restriction.

The energy stable BDF2 (BDF2\_ES) scheme \eqref{BDF2-CH_LOG-1} is expressed in 2D as
	\begin{eqnarray}
\phi_{i,j} - \frac{2\dt}{3}\, d_x\left({\cal M} \left(A_x\check\phi^{m+1}\right) D_x\mu \right)_{i,j} &
	\nonumber
	\\
- \frac{2\dt}{3}\, d_y\left( {\cal M} \left(A_y\check\phi^{m+1}\right) D_y\mu\right)_{i,j} &=& \frac{4}{3}\phi_{i,j}^m - \frac{1}{3}\phi_{i,j}^{m-1}  ,
	\label{disc-ch-bdf2-cs-1}
	\\
\mu_{i,j} - f_{c,\delta}'\left(\phi_{i,j}\right) +\epsilon^2\Delta_h\phi_{i,j}  +A \dt \Delta_h \phi_{i,j} &=& A \dt \Delta_h \phi_{i,j}^m -\check\phi_{i,j}^{m+1}  ,
	\label{disc-ch-bdf2-cs-2}
	\end{eqnarray}
where
	\[
\check{\phi}^{m+1}_{i,j} =  2 \phi^m_{i,j} - \phi^{m-1}_{i,j} .
	\]
The standard BDF2 scheme is
	\begin{eqnarray}
\phi_{i,j} - \frac{2\dt}{3}\, d_x\left({\cal M} \left(A_x \phi \right) D_x\mu \right)_{i,j} &
	\nonumber
	\\
- \frac{2\dt}{3}\, d_y\left( {\cal M} \left(A_y \phi \right) D_y\mu\right)_{i,j} &=& \frac{4}{3}\phi_{i,j}^m - \frac{1}{3}\phi_{i,j}^{m-1}  ,
	\label{disc-ch-bdf2-1}
	\\
\mu_{i,j} - f_{c,\delta}'\left(\phi_{i,j}\right) + \phi_{i,j} +\epsilon^2\Delta_h\phi_{i,j}  &=& 0 .
	\label{disc-ch-bdf2-2}
	\end{eqnarray}
As for the backward Euler scheme, solvability and stability are not unconditionally guaranteed for this scheme.

We use a nonlinear FAS multigrid method to solve all of the schemes efficiently. We give the details only for the (CS1) scheme, equations (\ref{disc-ch-1}) -- (\ref{disc-ch-2}). The details for the other methods are quite similar.  Our solver requires defining operator and source terms, which we do as follows.  Let $\bfphi = \left(\phi,\mu\right)^T$.  Define the nonlinear operator ${\bf N} = (N^{(1)},N^{(2)})^T$ as
	\begin{eqnarray}
N^{(1)}_{i,j}\left(\bfphi \right)  &=& \phi_{i,j} - \dt\, d_x\left(M\left(A_x\phi^m\right) D_x\mu \right)_{i,j} - \dt\, d_y\left(M\left(A_y\phi^m\right) D_y\mu \right)_{i,j} ,
	\label{op-disc-pfc-1}
	\\
N^{(2)}_{i,j}\left(\bfphi \right)  &=& \mu_{i,j} - f_{c,\delta}'\left(\phi_{i,j}\right) +\epsilon^2\Delta_h\phi_{i,j}  ,
	\label{op-disc-pfc-2}
	\end{eqnarray}
and the  source ${\bf S}= (S^{(1)},S^{(2)})^T$ as
	\begin{equation}
S^{(1)}_{i,j}\left(\bfphi \right) = \phi_{i,j} \ , \qquad S^{(2)}_{i,j}\left(\bfphi \right) = - \phi_{i,j}   .
	\label{source-disc-pfc}
	\end{equation}
Then, of course, Equations~(\ref{disc-ch-1}) -- (\ref{disc-ch-2}) are equivalent to ${\bf N}(\bfphi^{m+1}) = {\bf S}(\bfphi^m)$.  Notice that the operator ${\bf N}$ depends upon the time step $m$, because its definition involves the  solution $\phi^m$.

We mention that for the backward Euler (BE) scheme, the only difference in this decomposition is that
	\[
	N^{(2)}_{i,j}\left(\bfphi \right)  = \mu_{i,j} - f_{c,\delta}'\left(\phi_{i,j}\right) + \phi_{i,j} +\epsilon^2\Delta_h\phi_{i,j}  , \quad S^{(2)}_{i,j}\left(\bfphi \right) = 0.
	\]
The BDF2\_ES and BDF2 schemes are handled using similar considerations.

We will describe a somewhat standard nonlinear FAS multigrid scheme for solving the vector equation ${\bf N}(\bfphi^{m+1}) = {\bf S}(\bfphi^m)$.  Here we will sketch only the important points of the algorithm; the reader is referred to Trottenberg \emph{et al.}~\cite[Sec.~5.3]{trottenberg01}  and our paper~\cite{wise10} for complete details.  For this issue, we need to discuss a smoothing operator for generating \emph{smoothed} approximate solutions of  ${\bf N}(\bfphi) = {\bf S}$.  The action of this operator is represented as
	\begin{equation}
\widetilde\bfphi = \mbox{Smooth} \left(\lambda,\bfphi,{\bf N},{\bf S}
\right),
	\label{smooth-operator}
	\end{equation}
where $\bfphi$ is an approximate solution prior to smoothing, $\bar\bfphi$ is the smoothed approximation, and $\lambda$ is the number of smoothing sweeps. For smoothing we use a nonlinear Gauss-Seidel method with Red-Black ordering.  In what follows, to simplify the discussion, we give the details of the relaxation using the simpler lexicographic ordering.  Let $\ell$ be the index for the lexicographic Gauss-Seidel.  (Note that the smoothing index $\ell$ in the following should not be confused with the time step index $m$.)  Now we set
	\begin{eqnarray}
M^{\rm ew}_{i+\hf,j} := \mathcal{M}\left( A_x\phi^m_{i+\hf,j} \right) \ , \ \quad && M^{\rm ns}_{i,j+\hf} := \mathcal{M}\left( A_y\phi^m_{i,j+\hf}\right) \ .
	\nonumber
	\end{eqnarray}
The Gauss-Seidel smoothing is as follows: for every $(i,j)$,  stepping lexicographically from $(1,1)$ to $(N,N)$, find $\phi^{\ell+1}_{i,j}$, and $\mu^{\ell+1}_{i,j}$ that solve
	\begin{eqnarray}
&&\hspace{-.2in} \phi^{\ell+1}_{i,j}+\frac{\dt}{h^2}\left(M_{i+\hf,j}^{\rm ew} + M_{i-\hf,j}^{\rm ew} +M_{i,j+\hf}^{\rm ns} +M_{i,j-\hf}^{\rm ns}\right)\mu^{\ell+1}_{i,j}
	\nonumber
	\\
&=& S^{(1)}_{i,j}\left(\bfphi^m \right)
	\nonumber
	\\
& &+ \frac{\dt}{h^2}\Big(M_{i+\hf,j}^{\rm ew}\mu^{\ell}_{i+1,j} + M_{i-\hf,j}^{\rm ew}\mu^{\ell+1}_{i-1,j}  +M_{i,j+\hf}^{\rm ns}\mu^{\ell}_{i,j+1} +M_{i,j-\hf}^{\rm ns}\mu^{\ell+1}_{i,j-1}\Big) , \quad
	\label{smooth-1}
	\\
&&\hspace{-0.2in}\left(-f_{c,\delta}''\left(\phi^{\ell}_{i,j}\right)-\frac{4\epsilon^2}{h^2}\right)\phi^{\ell+1}_{i,j} + \mu^{\ell+1}_{i,j}
	\nonumber
	\\
&=& S^{(2)}_{i,j}\left(\bfphi^m \right)  + f_{c,\delta}'\left( \phi_{i,j}^\ell\right)- \phi_{i,j}^\ell f_{c,\delta}''\left( \phi_{i,j}^\ell\right)
	\nonumber
	\\
&&-\frac{\epsilon^2}{h^2}\left(\phi_{i+1,j}^{\ell}+\phi_{i-1,j}^{\ell+1}+\phi_{i,j+1}^{\ell}+\phi_{i,j-1}^{\ell+1} \right) .
	\label{smooth-2}
	\end{eqnarray}

Note that we have linearized the logarithmic term using a local Newton approximation, but otherwise this is a standard vector application of block Gauss-Seidel. The $2\times 2$ linear system defined by (\ref{smooth-1}) -- (\ref{smooth-2}) is unconditionally solvable (the determinant of the coefficient matrix is always positive in this case).  We use Cramer's Rule to obtain $\phi^{\ell+1}_{i,j}$ and $\mu^{\ell+1}_{i,j}$. However, we observe that it is not guaranteed that $-1<\phi_{i,j}^{\ell+1}<1$ for an arbitrary smoothing step.

The only difference for the backward Euler (BE) scheme is that second equation \eqref{smooth-2} in the block smoother is replaced by
	\begin{eqnarray*}
&&\hspace{-0.2in}\left(-f_{c,\delta}''\left(\phi^{\ell}_{i,j}\right)-\frac{4\epsilon^2}{h^2}\right)\phi^{\ell+1}_{i,j} + \mu^{\ell+1}_{i,j}
	\nonumber
	\\
&=& S^{(2)}_{i,j}\left(\bfphi^m \right)  + f_{c,\delta}'\left( \phi_{i,j}^\ell\right) - \phi_{i,j}^\ell - \phi_{i,j}^\ell f_{c,\delta}''\left( \phi_{i,j}^\ell\right)
	\nonumber
	\\
&&-\frac{\epsilon^2}{h^2}\left(\phi_{i+1,j}^{\ell}+\phi_{i-1,j}^{\ell+1}+\phi_{i,j+1}^{\ell}+\phi_{i,j-1}^{\ell+1} \right) .
	\end{eqnarray*}	

One full block Gauss-Seidel sweep has concluded when we have stepped lexicographically through all the grid points, from $(1,1)$ to $(N,N)$.  When $\lambda$ full smoothing sweeps has completed the vector result is labeled $\widetilde{\bfphi}$, as in Eq.~(\ref{smooth-operator}), and the action of the smoothing operator in \eqref{smooth-operator} is complete.

Multigrid works on a hierarchy of grids. We denote the grid level by the index $n$, where $n_{\min} \le n \le n_{\max}$, $n_{\max}$ is the index for the finest grid, and $n_{\min}$ is the index for the coarsest grid.  We need operators for communicating information from coarse levels to fine levels, and \emph{vice versa}.  By ${\bf I}_n^{n-1}$ we denote the restriction operator, which transfers fine grid functions, with grid index $n$, to the coarse grid, indexed by $n-1$.  By ${\bf I}_{n-1}^n$ we denote the prolongation operator, which transfers coarse grid functions (level $n-1$) to the fine grid (level $n$).   Here we work on cell-centered grids. The restriction operator is defined by cell-center averaging; for the prolongation operator we use piece-wise constant interpolation \cite[Sec.~2.8.4]{trottenberg01}.  The rest of the details of the nonlinear multigrid solver are similar to those given in~\cite{wise10}. The details are omitted for the sake of brevity.

	\subsection{Regularization of the logarithm in the multigrid solver}
	\label{subsec:delta-parameter}
	
We now give a very brief descussion on how to choose the regularization parameter $\delta$ for the logarithm. The regularization parameter must be chosen small enough so that the computed numerical solutions satisfy $-1+\delta < \phi_{i,j}^m < 1-\delta$, for any $i,j,m$. To understand the issue, consider the following simulation set-up: the common parameters are taken to be $\varepsilon = 5.0\times 10^{-3}$, ${\mathcal M}\equiv 1$, $T=1.0$, $L=1.0$, $h= 1/256$, $\dt = 1.0\times 10^{-3}$, $\tau = 1.0\times 10^{-9}$ (multigrid stopping tolerance). The initial conditions are $\phi_{i,j}^0 = 0.2+r_{i,j}$, where $r_{i,j} \in[-0.05,0.05]$ is a uniformly distributed random variable. We choose various values of the quench parameter $\theta_0$, the smallest being $\theta_0 = 2.0$ and the largest, $\theta_0 = 3.5$. As $\theta_0\to\infty$, the maxima and minima will tend to $1$ and $-1$, respectively, as the singular potential approaches the the obstacle potential.

We compute the maxima and minima of $\phi_{i,j}^m$ and report the values in Table~\ref{tab01}. Observe that for modest values of $\theta_0$, $\delta\in (0,1)$ can always be chosen so that $\nrm{\phi^m}_\infty < 1-\delta$. We point out, in particular, the case for which $\theta_0 = 3.0$. We have taken two different values of $\delta$, $1.0\times 10^{-3}$ and $1.0\times 10^{-5}$. The  computed solutions -- as well as the energies (not shown), which are decreasing at each time step -- for these two cases are the same up to round-off errors.

To be safe, in all of the computed solutions that follow, we use the smaller regularization parameter $\delta = 1.0\times 10^{-5}$. The same considerations are applied when picking the regularization parameter for 3-D simulations. In Figure~\ref{fig01}, we show  spinodal decomposition simulation using the parameters given in the caption. For $\delta$ sufficiently small, the computed solution stays well inside the interval $(-1+\delta, 1-\delta)$.
	
	\begin{figure}[h]
	\begin{center}
\includegraphics[width=\textwidth]{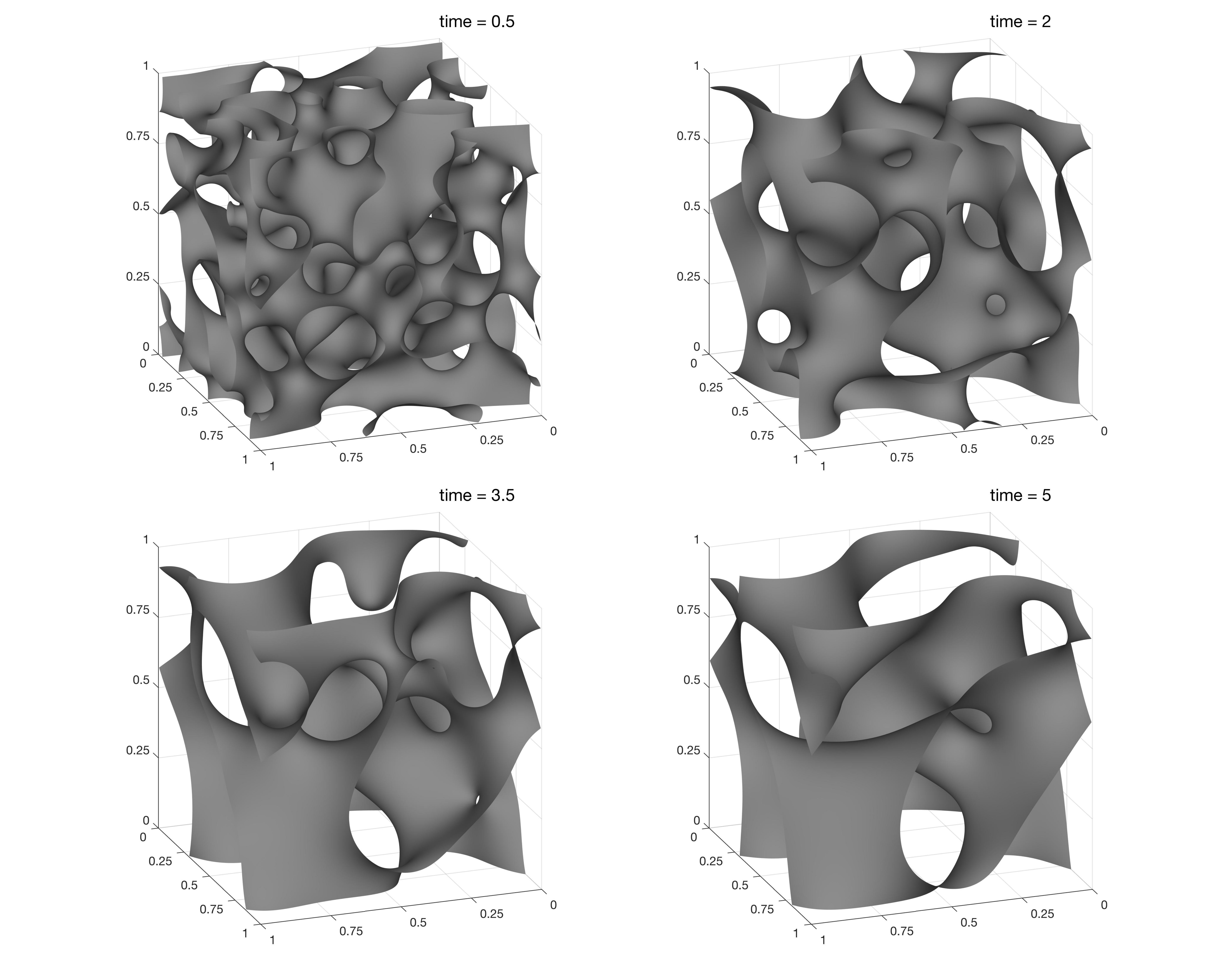}
\caption{Three-dimensional simulation. The parameters are  $L = L_x = L_y = L_z = 1.0$;  $\varepsilon = 5\times 10^{-3}$; $\theta_0 = 3.0$;  $\delta = 10^{-5}$;  $T = 5.0$;  $\tau = 10^{-8}$;  $\dt = 10^{-3}$, $h = \frac{1}{256}$. For initial data, $\phi^0_{i,j,k} = r_{i,j,k}$, where $r_{i,j,k}$ is a uniformly distributed random variable from the interval $[-0.05,0.05]$. The computed solution stays in the interval $[-0.996,0.996]$, well inside the interval $(-1+\delta, 1-\delta)$. This computation is done using the first-order convex-concave decomposition (CS1) scheme.}
	\label{fig01}
	\end{center}
	\end{figure}

	\begin{table}[!htb]
	\begin{center}
\caption{Maximum and minimum values of $\phi_{i,j}^k$ during spinodal decomposition, computed using the first-order convex-concave decomposition (CS1) scheme. The common parameters are $\varepsilon = 5.0\times 10^{-3}$, $T=1.0$, $L=1.0$, $h= 1/256$, $s = 1.0\times 10^{-3}$, $\tau = 1.0\times 10^{-9}$. The initial conditions are $\phi_{i,j}^0 = 0.2+r_{i,j}$, where $r_{i,j} \in[-0.05,0.05]$ is a uniformly distributed random variable. As $\theta_0$ becomes larger, the potential $f$ approaches the so-called obstacle potential, and the maxima and minima approach $+1$ and $-1$, respectively.  But, observe that the computed values stay well within the range $(-1+\delta, 1-\delta)$.}
	\label{tab01}
	\begin{tabular}{ccccc}
	\hline
$\theta_0$ & $\delta$ & $\lambda$ & ${\displaystyle \max_{i,j,k}\phi_{i,j}^k}$ & ${\displaystyle \min_{i,j,k}\phi_{i,j}^k}$
	\\
	\hline
$2.0$&$1.0\times 10^{-3}$&2  &$0.958159539817000$& $-0.969040263101000$
	\\
$2.5$&$1.0\times 10^{-3}$&2 & $0.986118743476000$ & $-0.990903230905000$
	\\
$3.0$ &$1.0\times 10^{-3}$& 2& $0.995203610902000$ & $-0.997255351479000$
	\\
$3.0$ & $1.0\times 10^{-5}$ & 2&$0.995203610606000$ & $-0.997255351459000$
	\\
$3.2$ & $1.0\times 10^{-5}$ &3&$0.996851091247000$ & $-0.998305411243000$
	\\
$3.5$ & $1.0\times 10^{-5}$ &3&$0.998298616883000$ & $-0.999144402772000$
	\\
	\hline
	\end{tabular}
	\end{center}
	\end{table}

	\subsection{Asymptotic ($\dt, h \to 0$) convergence test}
	\label{subsec-convergence}
Here we give a convergence test for the first-order convex-concave decomposition (CS1) scheme  method in 2-D. The initial condition for our convergence test is given by
	\begin{equation}
\phi(x,y,0) = 1.8\left(\frac{1-\cos\left(\frac{4x\pi}{3.2}\right)}{2}\right)  \left(\frac{1-\cos\left(\frac{2 y\pi}{3.2}\right)}{2}\right)-0.9 .
	\label{eqn:init}
	\end{equation}
The other parameters are as follows: (domain size) $L = L_x = L_y = 3.2$; (interfacial parameter) $\varepsilon = 0.2$; (mobility) ${\mathcal M}\equiv 1$; (quench parameter) $\theta_0 = 3.0$; ($\ln$ regularization parameter) $\delta = 1\times 10^{-5}$; (final time) $T = 0.4$; (solver stopping tolerance) $\tau = 10^{-9}$; (refinement path) $\dt = 0.4 h^2$. The test results are given in Table~\ref{tab02} and confirm the predicted accuracy: first order in time and second order in space. The other scheme are expected to exhibit optimal convergence rates, but the tests are not reported here for the sake of brevity.
	
	\begin{table}[!htb]
	\begin{center}
\caption{Errors and convergence rates.  The parameters are (domain size) $L = L_x = L_y = 3.2$; (interfacial parameter) $\varepsilon = 0.2$; (mobility) ${\mathcal M}\equiv 1$; (quench parameter) $\theta_0 = 3.0$; ($\ln$ regularization parameter) $\delta = 10^{-5}$; (final time) $T = 0.4$; (solver stopping tolerance) $\tau = 10^{-9}$; (refinement path) $\dt = 0.4 h^2$. The test results confirm the predicted accuracy: first order in time and second order in space.}
	\label{tab02}
	\begin{tabular}{cccc}
	\hline
$h_c$&$h_{f}$&$\nrm{\delta_\phi}_{2}$ & Rate
	\\
	\hline
$\frac{3.2}{16}$&$\frac{3.2}{32}$& $5.6689\times 10^{-2}$& --
	\\
$\frac{3.2}{32}$&$\frac{3.2}{64}$& $1.6071\times 10^{-2}$ & 1.819
	\\
$\frac{3.2}{64}$ &$\frac{3.2}{128}$& $4.1541\times 10^{-3}$ & 1.952
	\\
$\frac{3.2}{128}$ & $\frac{3.2}{256}$ &$1.0472\times 10^{-3}$ & 1.988
	\\
	\hline
	\end{tabular}
	\end{center}
	\end{table}

    \subsection{Algebraic convergence tests for the multigrid solver}
    \label{subsec-multigrid-convergence}

In this next test, we give some evidence that our multigrid solver for the first-order convex-concave decomposition (CS1) scheme has optimal or nearly optimal complexity. The solvers for the other schemes have similar, near-optimal performance.  We use the same test as in Section~\ref{subsec-convergence}.  The only difference is that for this test, we use a fixed time step size, $\dt=10^{-1}$ for all runs. We plot on a semi-log scale of the residual $\nrm{r^n}_{2}$ with respect to the multigrid iteration count $n$ at the 10th and final time step, \emph{i.e.}, $t=T=1.0$. The initial condition is defined in \eqref{eqn:init}, and the other parameters are as follows:  $L = L_x = L_y = 3.2$;  $\varepsilon = 0.2$; ${\mathcal M}\equiv 1$;   $\delta = 10^{-5}$. The quench parameter is varied, $\theta_0 = 3.5$, 3.0, and 2.0. The number of multigrid smoothing sweeps is held fixed at $\lambda = 2$. The multigrid stopping tolerance is taken to be $\tau = 10^{-9}$. The tests, reported in Figure~\ref{fig02}, indicate that the residual is reduced by nearly the same amount for each multigrid iteration. This is solid evidence for optimal or nearly optimal complexity. We do observe some minor degradation for larger values of $\theta_0$, which is expected, since the problem becomes increasingly stiff for larger values of $\theta_0$. In particular, the potential is approaching the super-singular obstacle potential in this limit.

    \begin{figure}[h]
	\begin{center}
\includegraphics[width=\textwidth]{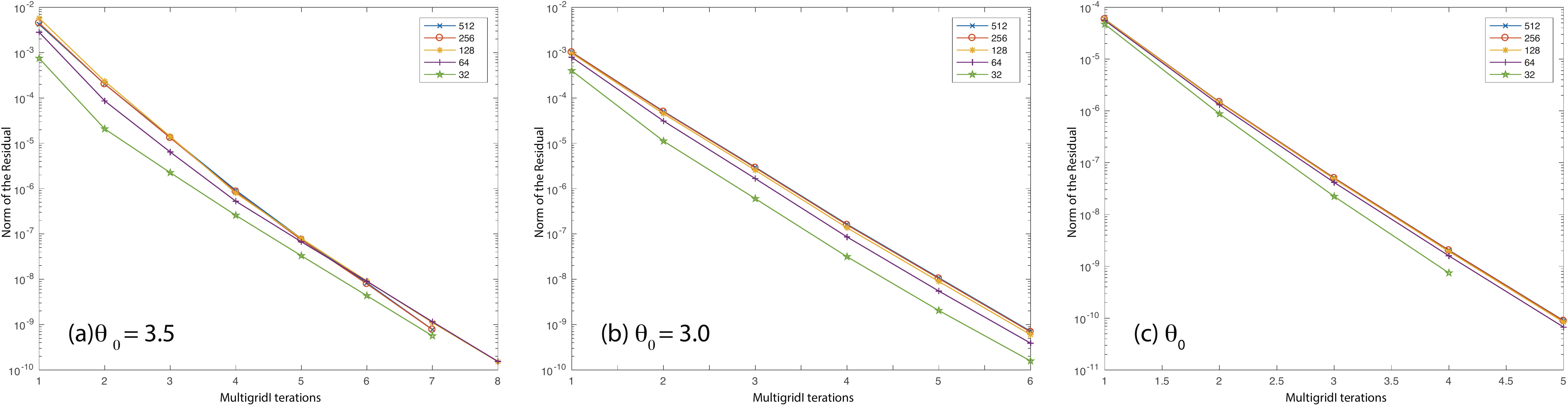}
\caption{Solver convergence (complexity) test for the problem defined in Section~\ref{subsec-convergence}.  We use a fixed time step size, $\dt=10^{-1}$ for all runs. We plot on a semi-log scale of the residual $\nrm{r^n}_{\ell^2}$ with respect to the multigrid iteration count $n$ at the 10th and final time step, \emph{i.e.}, $t=T=1.0$. The initial data is defined in \eqref{eqn:init}, and the other parameters are as follows:  $L = L_x = L_y = 3.2$;  $\varepsilon = 0.2$; ${\mathcal M}\equiv 1$;  $\delta = 10^{-5}$. The quench parameter is varied $\theta_0 = 3.5$, 3.0, and 2.0. The number of multigrid smoothing sweeps is held fixed at $\lambda = 2$. The multigrid stopping tolerance is taken to be $\tau = 10^{-9}$. We observe that the residual is decreasing by a nearly constant factor for each iteration. More iterations are required for larger values of $\theta_0$, as expected.}
	\label{fig02}
	\end{center}
	\end{figure}

	\begin{figure}[!htp]
	\begin{center}
\includegraphics[width=\textwidth]{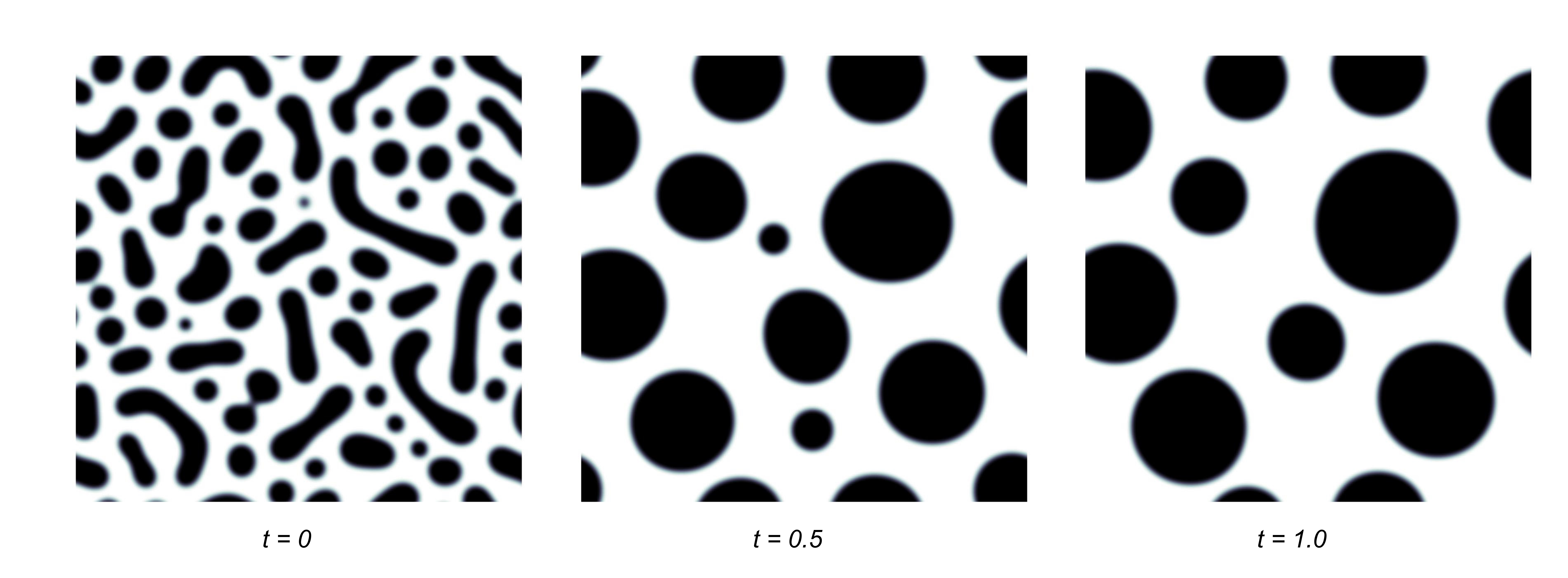}
	\caption{Initial data and high-resolution approximate solutions at $t=0.5$ and $t=1.0$. A high-resolution solution is computed using the BDF2 scheme~\eqref{disc-ch-bdf2-1} -- \eqref{disc-ch-bdf2-2} with the initial data shown in the figure ($t=0$). The parameters for the high-resolution approximation are $\dt = 1.0\times 10^{-5}$ and $h = 1.0/256$. The other parameters are $\Omega = (0,1.0)\times (0,1.0)$ and $\varepsilon = 5.0\times 10^{-3}$, $\theta_0 = 3.0$, $\delta = 1.0\times 10^{-5}$. Significant coarsening occurs between $t=0$ and $t=1.0$. In the simulation, we observe that, for the approximate solution, $0.99672 \ge \phi \ge -0.99821$.}
	\label{fig:compare}
	\end{center}
	\end{figure}

	\begin{table}[!htb]
	\begin{center}
\caption{The errors, average V-cycle iteration numbers for the FAS multigrid solvers, and the maximum values of $\phi$ for the  various schemes with fixed time and space step sizes $\dt = 1.0\times 10^{-4}$ and $h = 1.0/256$. The other parameters are $\Omega = (0,1.0)\times (0,1.0)$ and $\varepsilon = 5.0\times 10^{-3}$, $\theta = 3.0$, $\delta = 1.0\times 10^{-5}$. The ``errors,'' which are reported at times $t=0.1$, $t=0.5$ and $t=1.0$, are precisely the differences between the comparison approximations and the high-resolution target approximation computed using the BDF2 with the much smaller time step size $\dt = 5\times 10^{-6}$. See Figure~\ref{fig:compare}.}
	\label{tab:comparison-1}
	\begin{tabular}{|c|c|c|c|c|c|}
	\hline
Scheme & Error $t=0.1$& Error $t=0.5$ & Error $t=1$ & Ave.~Itr. & $ \displaystyle{\max_{i,j,k}\phi^k_{i,j}}$
	\\
	\hline
BDF2 & $2.2496{\rm e}-04$ & $9.3172{\rm e}-04$ & $5.2566{\rm e}-04$ & 4.6237 & 0.99671
	\\
	\hline
BDF2\_ES & $3.9485{\rm e}-02$ &$1.9105{\rm e}-01$ & $2.6703{\rm e}-01$ & 3.7373 & 0.99890
	\\
	\hline
\parbox{1.0cm}{BDF2\_ES \\ $A=0$} & $5.8446{\rm e}-03$ &$1.9113{\rm e}-02$ & $1.4204{\rm e}-02$ & 3.5390 & 0.99913
	\\
	\hline
BE & $2.7285{\rm e}-03$ & $8.4211{\rm e}-03$ & $5.7435{\rm e}-03$  & 6.5645 & 0.99668
	\\
	\hline
CS1 & $3.5965{\rm e}-01$ & $5.6166{\rm e}-01$ & $7.5356{\rm e}-01$ & 4.0003 & 0.99621
	\\
	\hline
	\end{tabular}
	\end{center}
	\end{table}
	
	\begin{table}[!htb]
	\begin{center}
\caption{The errors, average V-cycle iteration numbers for the FAS multigrid solvers, and the maximum values of $\phi$ for the various scheme with fixed time and space step sizes $\dt = 5.0\times 10^{-5}$ and $h = 1.0/256$. The other parameters are $\Omega = (0,1.0)\times (0,1.0)$ and $\varepsilon = 5.0\times 10^{-3}$, $\theta = 3.0$, $\delta = 1.0\times 10^{-5}$.  The ``errors,'' which are reported at times $t=0.1$, $t=0.5$ and $t=1.0$, are precisely the differences between the comparison approximations and the high-resolution target approximation computed using the BDF2 with the much smaller time step size $\dt = 5\times 10^{-6}$. See Figure~\ref{fig:compare}.}
	\label{tab:comparison-2}
	\begin{tabular}{|c|c|c|c|c|c|}
	\hline
Scheme & Error $t=0.1$& Error $t=0.5$ & Error $t=1$ & Ave.~Itr. & $ \displaystyle{\max_{i,j,k}\phi^k_{i,j}}$
	\\
	\hline
BDF2 & $5.7762{\rm e}-05$ & $2.3749{\rm e}-04$ & $1.3267{\rm e}-04$ & 3.49560 & 0.99671
	\\
	\hline
BDF2\_ES & $1.0079{\rm e}-02$ &$1.1464{\rm e}-02$ &  $7.6568{\rm e}-03$ & 2.70300 & 0.99668
	\\
	\hline
\parbox{1cm}{BDF2\_ES \\ $A=0$} & $1.6392{\rm e}-03$ &$5.1465{\rm e}-03$ & $4.0927{\rm e}-03$ & 2.6401 & 0.99671
	\\
	\hline
BE & $1.3510{\rm e}-03$ & $4.1560{\rm e}-03$ & $2.8182{\rm e}-03$  & 3.63475 & 0.99670
	\\
	\hline
CS1 & $1.4975{\rm e}-01$ & $2.7690{\rm e}-01$ & $3.5650{\rm e}-01$ & 2.78145 & 0.99628
	\\
	\hline
	\end{tabular}
	\end{center}
	\end{table}

	\subsubsection{Initial data and a high-resolution approximate solution at $t=1$}
A high-resolution solution is computed using the BDF2 scheme~\eqref{disc-ch-bdf2-1} -- \eqref{disc-ch-bdf2-2} with the initial data shown in Figure~\ref{fig:compare} ($t=0$). The parameters for the  approximation are $\dt = 5\times 10^{-6}$ and $h = 1.0/256$. The physical parameters are $\Omega = (0,1)^2$,  $\varepsilon = 5.0\times 10^{-3}$, ${\mathcal M}\equiv 1$; $\theta_0 = 3.0$, and $\delta = 1.0\times 10^{-5}$. Note that the time step size $\dt = 5\times 10^{-6}$ is 10 times smaller than what will be used in the comparison tests, and we will treat the approximation obtained here as the target solution. We point out that computing the target solution with the slightly larger time step of $\dt = 1\times 10^{-5}$ does not change the results presented in Tables~\ref{tab:comparison-1} and \ref{tab:comparison-2} in any significant way.

	\subsubsection{Comparison results}
For the comparison computations we use the same parameters as above -- $h = 1.0/256$, $\Omega = (0,1)^2$,  $\varepsilon = 5.0\times 10^{-3}$, ${\mathcal M}\equiv 1$; $\theta_0 = 3.0$, $\delta = 1.0\times 10^{-5}$ --  but we use larger time step sizes: $\dt = 1.0\times 10^{-4}$ (Table~\ref{tab:comparison-1}) and $\dt = 5.0\times 10^{-5}$ (Table~\ref{tab:comparison-2}).  To solve all of the schemes, we employ the FAS multigrid methods detailed above. The results of the tests are reported in Tables~\ref{tab:comparison-1} and \ref{tab:comparison-2}, and they paint a complicated picture. The BDF2 scheme shows excellent accuracy and efficiency. Based on our experience, this method is the most accurate of the four that have been test, which is why it is used to generate our target solution. Our new BDF2\_ES scheme is slightly more efficient, but not nearly as accurate. When the stabilization parameter is set to zero ($A = 0$), its accuracy increases significantly, but its provable stability is lost.

The first-order convex-concave decomposition scheme is the worst in the tests for accuracy, but the second best in efficiency per step. The worst in efficiency per time step is the backward Euler scheme; like the BDF2 scheme, it does not have a convex structure. But, like pure BDF2, the fully implicit backward Euler has very good accuracy, better than the energy stabilized BDF2 scheme with the stabilization parameter set to zero.

All of the schemes are positivity preserving, as long as they are solvable. Even though we did not prove this claim for the fully implicit schemes, such a fact can be established in our theory, though the details are significantly more complicated and are skipped in this presentation.

	\section{Conclusion remarks}
	\label{sec:conclusion}

In this paper we have presented and analyzed two positivity preserving, energy stable finite difference schemes for the Allen Cahn/Cahn-Hilliard model with a logarithmic Flory Huggins energy potential, including both the first and second order temporal accuracy. In particular, the singular nature of the logarithmic term around the values of $-1$ and 1 prevents the numerical solution from reaching these singular values, and this subtle fact indicates that the proposed numerical algorithm has a unique solution with preserved positivity for the logarithmic arguments.  In turn, the numerical scheme is always well-defined, as long as the numerical solution stays bounded at the previous time step, which is natural. And also, an unconditional energy stability has been theoretically justified; in particular, an artificial Douglas-Dupont regularization term is added in the second order BDF scheme to ensure the energy stability. In addition, an optimal rate convergence in the $\ell^\infty (0,T; H_h^{-1}) \cap \ell^2 (0,T; H_h^1)$ norm has been established for both the first and second order accurate schemes. An efficient multigrid solver is applied in the practical implementation, and some numerical results are presented, which demonstrate the robustness and efficiency of the numerical solver.

	\section*{Acknowledgment}
This work is supported in part by the grants NSFC 11671098, 11331004, 91630309, a 111 project B08018 (W.~Chen), NSF DMS-1418689 (C.~Wang), NSF DMS-1715504 and Fudan University start-up (X.~Wang) and  NSF DMS-1719854 (S.~Wise). C.~Wang also thanks the Key Laboratory of Mathematics for Nonlinear Sciences, Fudan University, and Shanghai Center for Mathematical Sciences, for support during his visit. During the finalization of the manuscript, S.~Wise was partially supported by the Techniche Universit\"{a}t, Dresden (TUD), as a senior Dresden Fellow and by Oak Ridge National Laboratory (ORNL) while this work was being completed. S.~Wise thanks TUD and Prof.~Axel Voigt for the generous support and hospitality and thank Cory Hauck (ORNL) for support and discussions on this and related topics.

	\bibliographystyle{plain}
	\bibliography{chlog_3}

\end{document}